\documentclass[12pt]{amsart}

\usepackage{graphicx}

\usepackage[latin1]{inputenc}

\usepackage[margin=3cm,bottom=2.5cm,top=2.5cm,footskip=1.3cm]{geometry}

\usepackage{color}

\usepackage{hyperref}
\hypersetup{colorlinks=true}

\usepackage{amsmath,amsfonts,amsthm,amssymb,mathtools} 

\newcommand\N{{\mathbb N}} \newcommand\R{{\mathbb R}}
\newcommand\Z{{\mathbb Z}} \newtheorem{thm}{Theorem}

\newtheorem{rmq}{Remark}[section] \newtheorem{lemma}{Lemma} %
 \newtheorem{prop}{Proposition}

\newcommand{\DomainXYT}{\mathcal{R}}

\begin{document}

 \title{New counterexamples to Strichartz estimates for the wave equation on a 2D model convex domain}

\author{Oana Ivanovici}
\address{Sorbonne Université, CNRS, Laboratoire Jacques-Louis Lions, LJLL, F-75005 Paris, France} \email{oana.ivanovici@sorbonne-universite.fr}

  \author{Gilles Lebeau}
  \address{Universit\'e Côte d'Azur, CNRS, Laboratoire JAD, France} \email{gilles.lebeau@univ-cotedazur.fr}

  \author{Fabrice Planchon}
  \address{ Sorbonne Université, CNRS, Institut Mathématique de Jussieu-Paris Rive Gauche, IMJ-PRG F-75005 Paris, France}
  \email{fabrice.planchon@sorbonne-universite.fr} 
  
      \thanks{
    \\
 O.Ivanovici and F. Planchon were supported by ERC grant ANADEL 757 996.
    } 
%\date{\today}

\begin{abstract} We prove that the range of Strichartz estimates on a model 2D convex domain may be further restricted compared to the known counterexamples from \cite{doi,doi2}. Our new family of counterexamples is built on the parametrix construction from \cite{Annals} and revisited in \cite{ILP3}. Interestingly enough, it is sharp in at least some regions of phase space.
\end{abstract}
\maketitle

\section{Introduction and main results}

Let us consider the wave equation on a domain $\Omega$ with boundary
$\partial \Omega$ ,
\begin{equation} \label{WE} \left\{ \begin{array}{l} (\partial^2_t-
\Delta) u(t, x)=0, \;\; x\in \Omega \\ u|_{t=0} = u_0,\,\,\, \partial_t
u|_{t=0}=u_1,\\ Bu=0,\quad x\in \partial\Omega.
 \end{array} \right.
 \end{equation}
Here,  $\Delta$ stands for the Laplace-Beltrami operator on
$\Omega$. If $\partial\Omega\neq \emptyset$, the boundary condition
could be either Dirichlet ($B$ is the identity map:
$u|_{\partial\Omega}=0$) or Neumann ($B=\partial_{\nu}$ where $\nu$ is
the unit normal to the boundary.) We will take $B=Id$ but the argument may be adapted to Neumann.

The so called Strichartz estimates aim at quantifying dispersive
properties of the solutions to this linear wave equation: for given
data in the natural energy space, the solution will have better decay
for suitable time averages.  This is of value for several applications, of which we quote only two:
\begin{itemize}
\item nonlinear problems, where Strichartz may be used as a tool to improve on Sobolev embeddings and allow for better nonlinear mapping properties of solutions;
\item  localization properties of (clusters of) eigenfunctions of the Laplacian (through square function estimates for the wave equation which are closely related to Strichartz estimates).
\end{itemize}

On any compact Riemannian manifold with empty boundary, the solution to \eqref{WE} is such that, at least for a suitable $t_{0}<+\infty$, for all $h<1$,
\begin{equation}\label{stricrd} h^{\beta}\|\chi(hD_t)u\|_{L^q([-t_{0},t_{0}],
L^r)}\leq C\left(\|u(0,x)\|_{L^2}+\|hD_t u\|_{L^2}\right),
\end{equation} 
where $\chi\in C^{\infty}_0$ is a smooth truncation in a neighborhood of $1$. Let $d$ be the spatial dimension of $\Omega$, then rescaling dictates that $\beta=d\left(\frac 12-\frac 1r\right)-\frac 1q$, where $(q,r)$ is a so-called admissible pair:
\begin{equation}\label{adm} 
\frac 1q\leq \frac{(d-1)} 2 \left(\frac{1}{2}-\frac{1}{r}\right),\quad  q>2.
\end{equation}
On non compact manifolds, one would have to assume suitable geometric assumptions to allow these estimates to hold globally: when \eqref{stricrd} holds for $t_{0}=+\infty$, it is said to be a global in time Strichartz estimate. For $\Omega=\mathbb{R}^d$ with flat metric, the solution  $u_{\mathbb{R}^d}(t,x)$ to
\eqref{WE} with initial data $(u_0=\delta_{x_{0}}, u_1=0)$ has an explicit representation formula 
 \[ u_{\mathbb{R}^d}(t,x)=\frac{1}{(2\pi)^d}\int
\cos(t|\xi|)e^{i(x-x_{0})\xi}d\xi
 \]
and by usual stationary phase methods one gets dispersion:
\begin{equation}\label{disprd}
\|\chi(hD_t)u_{\mathbb{R}^d}(t,.)\|_{L^{\infty}(\mathbb{R}^d)}\leq
C(d) h^{-d}\min\{1, (h/t)^{\frac{d-1}{2}}\}.
\end{equation}
Interpolation between  \eqref{disprd} and energy estimates, together with a duality argument, routinely provides
\eqref{stricrd} (\cite{Stri77}, \cite{Pe84}, \cite{GV85}). On any (compact) Riemannian manifold without boundary $(\Omega,g)$ one may follow the same path, replacing the exact formula by a parametrix, which may be constructed locally (in time and space) within a small ball, thanks to finite speed of propagation (\cite{Kapi89}, \cite{MSS93}). By routine computations, one may deduce from the semi-classical estimate \eqref{stricrd} standard estimates involving mixed Lebesgue-Besov norms on the left handside and Sobolev spaces on the right handside; these are better suited to dealing with nonlinear problems. 

On a manifold with boundary, the geometry of light rays becomes much more complicated, and one may no longer think that one is slightly bending flat trajectories. There may be gliding rays (along a convex boundary) or grazing rays (tangential to a convex obstacle) or combinations of both. Strichartz estimates outside a strictly convex obstacle were obtained in \cite{smso95} and turned out to be similar to the free case (see \cite{ildispext} for the more complicated case of the dispersion). Strichartz estimates with losses were obtained later on general domains, \cite{blsmso08}, using short time parametrices constructions from \cite{smso06}, which in turn were inspired by works on low regularity metrics \cite{tat02}. Most of these works focus either on compact domains with boundary or exterior domains, although one may combine existing results to deal with unbounded domains with suitable control over geometry at infinity.

In our work \cite{Annals}, a parametrix for the wave equation inside a model of strictly convex domain was constructed that provided optimal decay estimates, uniformly with respect to the distance of the source to the boundary, over a time length of constant size. This involves dealing with an arbitrarily large number of caustics and retain control of their order. Our dispersion estimate from \cite{Annals} is optimal and immediately yields by the usual argument Strichartz estimates with a range of pairs $(q,r)$ such that
\begin{equation}\label{adm-1/4} 
\frac 1q\leq \left(\frac{(d-1)} 2 -\frac 1 4\right)\left(\frac{1}{2}-\frac{1}{r}\right),\quad  q>2\,
\end{equation}
where, informally, the new $1/4$ factor, when compared to \eqref{adm}, is related to the $1/4$ loss in the dispersion estimate from \cite{Annals}, when compared to \eqref{disprd}. On the other hand, earlier works \cite{doi,doi2} proved that Strichartz estimates on strictly convex domains can hold only if, when $r>4$, $(1/q,1/r)$ are below a line connecting the pair $(1/q_{4},1/4)$ (from free space) and $(1/q_{\infty},0)$  such that
\begin{equation}\label{adm-1/4bis} 
\frac 1 {q_{4}}=\frac{(d-1)} 2\left(\frac 1 2 -\frac 1 4\right) \,\,\text{ and }\,\, \frac 1 {q_{\infty}}= \left(\frac{(d-1)} 2 -\frac {1}{12}\right)\frac{1}{2}\,.
\end{equation}
We will restate the exact result later on as we provide a simplified proof for it. Our main purpose in the present work is to improve upon the negative results in dimension $d=2$; improvements on the positive side were obtained in \cite{ILP3}. In particular, for suitable microlocalized solutions we close the gap between known estimates and known counterexamples, providing a
near complete picture in a specific location in phase space. Before stating our main results, we start by describing our convex model domain. Our Friedlander model is the half-space, for $d\geq 2$,
 $\Omega_d=\{(x,y)| x>0, y\in\mathbb{R}^{d-1}\}$ with the metric $g_F$ inherited from the following Laplace operator,
 $\Delta_F=\partial^2_x+(1+x)\Delta_{\mathbb{R}^{d-1}_y}$, with Dirichlet boundary condition on $x=0$. The domain $(\Omega_d,g_F)$ is easily seen to be a strictly convex set, as a first order approximation of the unit disk $D(0,1)$ in polar coordinates $(r,\theta)$: set $r=1-x/2$, $\theta=y$.

We start by stating our results for $d=2$ and later provide the general statement in higher dimensions, using the same reduction as \cite{doi2} to take advantage of the 2D setting.
\begin{thm}
  \label{thm-1}
 Strichartz estimates \eqref{stricrd} may hold true on the domain $(\Omega_{2},g_{F})$ only if possible pairs $(q,r)$ are 
such that
\begin{equation}\label{adm-1/12} 
\frac 1q\leq \left(\frac{1} 2 -\frac 1 {10}\right)\left(\frac{1}{2}-\frac{1}{r}\right)\,.
\end{equation}
In particular, for $r=+\infty$, we have $q\geq 5$.
\end{thm}
\begin{rmq} Theorem \ref{thm-1} improves on the results from
  \cite{doi}: the range of admissible pairs is further restricted as $1/12$ is replaced by $1/10$ in the admissibility condition. Moreover, we no longer have a restricted range of $r$, unlike \cite{doi}.
\end{rmq}
In \cite{ILP3}, we obtained the following positive results:
\begin{thm}[\cite{ILP3}]
 Strichartz estimates \eqref{stricrd} hold true on $(\Omega_{2},g_{F})$ for $(q,r)$ such that
\begin{equation}\label{adm-1/9autrepapier} 
\frac 1q\leq \left(\frac{1} 2 -\frac 1 {9}\right)\left(\frac{1}{2}-\frac{1}{r}\right)\,.
\end{equation}
In particular, for $r=+\infty$, we have $q\geq 5+1/7$.
\end{thm}
A gap remains between negative results ($1/10$ in \eqref{adm-1/12}) and positive results ($1/9$ in \eqref{adm-1/9autrepapier}).
\begin{rmq} Besides the full Laplacian, both $-\partial_{y}^{2}$ and $x+(-\partial^{2}_{y})^{-1}(-\partial_{x}^{2})$ commute with the wave flow. In \cite{ILP3} we obtain that, whenever the data is moreover restricted to $|\partial_{y}|\sim h^{-1}$ and $x+(-\partial^{2}_{y})^{-1}(-\partial_{x}^{2})\sim h^{1/3}$, then Strichartz estimates hold for $q>5$. Hence, in this region of phase space, Theorem \ref{thm-1} is optimal except for the endpoint $q=5$.
\end{rmq}
Counterexamples in \cite{doi} were constructed by carefully
propagating a cusp starting in a suitable position around $(a,0)\in \Omega_{2}$, with $a\sim
h^{1/2}$. Here we start with a smoothed out cusp, which may be seen as a wave packet around $a\sim h^{1/3}$
and let it propagate, estimating the resulting solution with the
parametrix and proving it saturates the bound with a set of exponents satisfying \eqref{adm-1/12}. Our special solution may be seen as a sum of consecutive wave reflections, and at any given point in space-time we see at most one of these waves. Each wave has its peak around a specific location related to the number of reflections, and we can estimate the area (in $(x,y)$) where the amplitude of the wave remains close to its peak value, allowing to lower bound any of its physical Lebesgue norms. The time norm is then estimated taking advantage of the separation between any two different wave reflections.

From the 2D construction, we can easily follow the strategy from \cite{doi2}, and construct a good approximate $d-$dimensional wave by tensor product: retain our 2D wave in a given spatial tangential direction and multiply by a Gaussian of width $h^{1/2}$ in all other tangential directions. Such a wave packet will then provide a special solution that saturates some $d-$dimensional estimates. However, it turns out that we do not recover better counterexamples than the ones from \cite{doi2}: in fact, we recover the exact same set of exponents, albeit for a slightly different class of examples. As such we state the result and its proof for the sake of completeness as well as providing a much simpler argument than both \cite{doi,doi2}.
\begin{thm}
  \label{thm-2}
For $d=3,4,5$, Strichartz estimates \eqref{stricrd} may hold true on the domain $(\Omega_{d},g_{F})$ only if possible pairs $(q,r)$ are 
such that 
\begin{equation}\label{adm-1/10d} 
\frac 1q\leq \left(\frac{d-1} 2 -\frac {1-4/r}{12-24/r}\right)\left(\frac{1}{2}-\frac{1}{r}\right)\,.
\end{equation}
\end{thm}
Note that we get the same dimension restriction out of necessity: we have an additional condition $r\geq 4$ that restricts meaningful ranges to lower dimensions.

Finally, we comment on dealing with only a model case: Theorem \ref{thm-1} should be seen as a better version of the results from \cite{doi}. Counterexamples from \cite{doi} do not directly provide counterexamples for a generic convex domain, and it required further treatment in \cite{doi2}. We believe that the present construction is a lot simpler than that of \cite{doi}, mostly thanks to the use of the exact parametrix from \cite{ILP3}. As such, constructing a generic counterexample will be easier, using in turn the parametrix obtained in \cite{ILLP} and following the present work as a blueprint. In fact, we suggest to any interested reader to start with the present paper, followed by \cite{ILP3}, \cite{ILLP} and only afterward, if inclined to, \cite{doi}, \cite{doi2} and \cite{Annals}.

In the remaining of the paper, $A\lesssim B$ means that there exists a constant $C$ such that $A\leq CB$ and this constant may change from line to line but is independent of all parameters. It will be explicit when (very occasionally) needed. Similarly, $A\sim B$ means both $A\lesssim B$ and $B\lesssim A$.

\subsection*{Acknowledments}
The authors thank all referees for their careful reading and constructive remarks and suggestions.
\section{The half-wave propagator: spectral analysis and parametrix construction}

\subsection{Digression on Airy functions}
Before dealing with the Friedlander model, we recall a few notations, where $Ai$ denotes the standard Airy function (see e.g. \cite{AFbook} for well-known properties of the Airy function), $  Ai(x)=\frac 1 {2\pi } \int_{\R} e^{ i  (\frac{\sigma^{3}}{3}+\sigma x)} \,d\sigma$: define
\begin{equation}
  \label{eq:Apm}
  A_\pm(z)=e^{\mp i\pi/3} Ai(e^{\mp i\pi/3} z)=-e^{\pm 2i\pi/3} Ai(e^{\pm 2i\pi/3} (-z))\,,\,\,\text{ for } \,
  z\in \mathbb{C}\,,
\end{equation}
then one checks that $Ai(-z)=A_+(z)+A_-(z)$ (see \cite[(2.3)]{AFbook}). The next lemma is proved in the Appendix and requires the classical notion of asymptotic expansion: a function $f(w)$ admits an asymptotic expansion for $w\rightarrow 0$ when there exists a (unique) sequence $(c_{n})_{n}$ such that, for any $n$, $\lim_{w\rightarrow 0} w^{-(n+1)}(f(w)-\sum_{0}^{n} c_{n} w^{n})=c_{n+1}$. We will denote $ f(w)\sim_{w} \sum_{n} c_{n} w^{n}$.
\begin{lemma}
  \label{lemL}
Define 
%\begin{equation}
%  \label{eq:Lom}
 $L(\omega)=\pi+i\log \frac{A_-(\omega)}{A_+(\omega)}$ % \,,\,\,\text{
for $\omega \in \R$,
%\end{equation}
then $L$ is real analytic and strictly increasing. We also
have
\begin{equation}
  \label{eq:propL}
  L(0)=\pi/3\,,\,\,\lim_{\omega\rightarrow -\infty} L(\omega)=0\,,\,\,
  L(\omega)=\frac 4 3 \omega^{\frac 3 2}+\frac{\pi}{2}-B(\omega^{\frac 3
    2})\,,\,\,\text{ for } \,\omega\geq 1\,,
\end{equation}
with the following asymptotic expansion for $B$, with $b_{1}>0$ and $(b_{k})_{k\geq 1}\in \R^{\N}$,
\begin{equation}
  \label{eq:B}
  B(u)\sim_{\frac 1u } \sum_{k\geq 1} b_k u^{-k}\,.
\end{equation}
Finally, let $\{-\omega_k\}_{k\geq 1}$ denote the zeros of the Airy function in decreasing order,
\begin{equation}
  \label{eq:propL2}
  L(\omega_k)=2\pi k\,,\,\, %\text{ and }
  L'(\omega_k)=2\pi \int_0^\infty Ai^2(x-\omega_k) \,dx\,.
\end{equation}
\end{lemma}

\subsection{Spectral analysis of the Friedlander model}

Recall $\Omega_2=\{(x,y)\in\mathbb{R}^2|, x>0,y\in\mathbb{R}\}$ and
$\Delta_F=\partial^{2}_{x}+(1+x)\partial^{2}_{y}$ with Dirichlet boundary condition. After a Fourier transform in the $y$ variable, the operator $-\Delta_{F}$ is now
$-\partial^2_x+(1+x)\theta^2$. For  $\theta\neq 0$, this operator is a positive self-adjoint operator 
on $L^2(\mathbb{R}_+)$, with compact resolvent and we have explicit eigenfunctions and eigenvalues (the proof of the next lemma is, again, postponed to the Appendix):

\begin{lemma}\label{lemorthog}
There exist orthonormal eigenfunctions $\{e_k(x,\theta)\}_{k\geq 0}$ with their corresponding eigenvalues $\lambda_k(\theta)=|\theta|^2+\omega_k|\theta|^{4/3}$, which form an Hilbert basis of $L^{2}(\mathbb{R}_{+})$.  These eigenfunctions have an explicit form
\begin{equation}\label{eig_k}
 e_k(x,\theta)=\frac{\sqrt{2\pi}|\theta|^{1/3}}{\sqrt{L'(\omega_k)}}
Ai\Big(|\theta|^{2/3}x-\omega_k\Big),
\end{equation}
where $L'(\omega_k)$ is given by \eqref{eq:propL2}, which yields
$\|e_k(.,\theta)\|_{L^2(\mathbb{R}_+)}=1$.
\end{lemma}

In a classical way, for $a>0$, the Dirac distribution $\delta_{x=a}$ on $\mathbb{R}_+$ may be decomposed as
\[
 \delta_{x=a}=\sum_{k\geq 1} e_k(x,\theta)e_k(a,\theta).
\]
Then if we consider a data at time $t=s$ such that
$u_0(x,y)=\psi(hD_y)\delta_{x=a,y=b}$, where $h\in (0,1)$ is a small parameter and
 $\psi\in C^{\infty}_0([\frac 12,2])$, we can write the (localized in $\theta$) Green function associated to the half-wave operator on $\Omega_{2}$ as
 \begin{equation}
\label{greenfct} G^{\pm}_{h}((x,y,t),(a,b,s))=\sum_{k\geq 1}
\int_{\mathbb{R}}e^{\pm i(t-s)\sqrt{\lambda_k(\theta)}}
e^{i(y-b)\theta} 
 \psi(h\theta)e_k(x,\theta)e_k(a,\theta)d\theta\,.
 \end{equation}
\subsection{Airy-Poisson formula}
We briefly recall a variant of the Poisson summation formula, introduced to deal with a parametrix construction for the general case of a generic strictly convex domain in \cite{ILLP} and used in \cite{ILP3} to improve Strichartz estimates in the model case. It will turn out to be crucial to analyze the spectral sum defining $G^{\pm}_{h}$ and map it to a sum over reflections of waves.
\begin{lemma}
  \label{AiryPoisson}
  In $\mathcal{D}'(\R_\omega)$, one has
  \begin{equation}
    \label{eq:AiryPoisson}
        \sum_{N\in \Z} e^{-i NL(\omega)}= 2\pi \sum_{k\in \N^*} \frac 1
    {L'(\omega_k)} \delta(\omega-\omega_k)\,.
  \end{equation}
In other words, for $\phi(\omega)\in C_{0}^{\infty}$, 
  \begin{equation}
    \label{eq:AiryPoissonBis}
        \sum_{N\in \Z} \int e^{-i NL(\omega)} \phi(\omega)\,d\omega = 2\pi \sum_{k\in \N^*} \frac 1
    {L'(\omega_k)} \phi(\omega_k)\,.
  \end{equation}

\end{lemma}
The Lemma is easily proved using the usual Poisson summation formula
followed by the change of variable $x=L(\omega)$ and we provide details in the Appendix.

\section{Counterexamples}
As recalled in the introduction, counterexamples in \cite{doi} were constructed by carefully
propagating a cusp starting at a distance $a\sim h^{1/2}$ from the boundary. In this section, $a$ is a parameter to be optimized later on, which is to be thought as the distance between the boundary and the peak value of the data (and later, repeatedly in time, of the solution itself). Recall that a (2D) Strichartz estimate is
\begin{equation}
  \label{eq:31}
  \| u\|_{L^{q}([0,t_{0}],L^{r}(\Omega))}\lesssim h^{-\beta} \|u_{0}\|_{L^{2}(\Omega)}\,,
\end{equation}
where $\beta=d(1/2-1/r)-1/q$ with $d=2$ (scaling condition). We also define $\alpha$ to be such that $1/q=\alpha(1/2-1/r)$ and recall that in free space, $\alpha=(d-1)/2=1/2$.
\subsection{Rescaled variables}
Let $a$ be small enough, such that $ h^{2/3}\ll a \ll 1$. From our knowledge from the parametrix construction in \cite{Annals} (see also \cite{ILP3}), where the source point is $(x=a,y=0)$, we rescale as follows:  set $\lambda=a^{3/2}/h$ and let $M_a=a^{-1/2}$,
\begin{equation}
  \label{eq:38}
  t=a^{1/2} T\,,\,\, x=aX\,,\,\,y=-t\sqrt{1+a}+a^{3/2}Y\,,\,\, U(T,X,Y)=u(t,x,y)\,.
\end{equation}
If $F(X,Y)=f(aX,a^{3/2}Y-T\sqrt a \sqrt{1+a})$, then
\begin{equation}
  \label{eq:39}
  \| F(X,Y)\|_{L^{r}_{X>0,Y}}=a^{-5/(2r)} \|f\|_{L^{r}_{x>0,y}}
\end{equation}
and 
\begin{equation}
  \label{eq:40}
   \| U(T, X,Y)\|_{L^{q}([0,M_a],L^{r})}=a^{-1/(2q)-5/(2r)} \|u\|_{L^{q}(0,1;L^{r})}\,.
\end{equation}
Since $h=M_a^{-3} \lambda^{-1}$, in rescaled variables, \eqref{eq:31} becomes
\begin{equation}
  \label{eq:41}
  M_a^{-1/q-5/r} \| U\|_{L^{q}([0,M_a],L^{r})}\lesssim (\lambda M_a^{3})^{1-1/q-2/r} a^{5/4} \| U_{0}\|_{L^{2}}
\end{equation}
hence we are reduced to 
\begin{equation}
  \label{eq:41bis}
 \| U\|_{L^{q}([0,M_a],L^{r})}\lesssim \lambda^{1-1/q-2/r} M_a^{1/2-1/r-2/q} \| U_{0}\|_{L^{2}}\,.
\end{equation}

\subsection{Setup for the parametrix}
Let us consider our model equation,
\begin{equation}
  \label{eq:42}
(  \partial^{2}_{t} -(\partial^{2}_{x}+(1+x)\partial^{2}_{y})) u(t,x,y)=0 \text{ on } x\geq 0, \,y\in\mathbb{R}
\end{equation}
with Dirichlet boundary condition $u_{|x=0}=0$. We will seek solutions $u$ under the following form, where the Fourier variable $\theta$ associated to $y$ is rescaled with $\eta=h\theta$,
\begin{equation}
  \label{eq:43}
  u(t,x,y)=\frac 1{2\pi h} \int e^{i\frac \eta h y} v(t,x,\eta/h) \psi(\eta)\,d\eta\,,
\end{equation}
where $\psi\in C^{\infty}_{0}$, $\psi=1$ for $3/4\leq \eta\leq 3/2$ and $\psi=0$ outside $[\frac 12,2]$. Therefore, as a function of $y$, $u$ is band-limited and its Fourier variable $\theta\sim 1/h$. If we set $\hbar=h/\eta$ and  $v_{\hbar}(t,x)=v(t,x,1/\hbar)$, $v_{\hbar}$ is a solution to
 \begin{equation}
   \label{eq:44}
(   \hbar^{2}\partial^{2}_{t}+(-\hbar^{2}\partial^{2}_{x}+(1+x)) v_{\hbar}(t,x)=0 \text{ for } x\geq 0,
 \end{equation}
with $v_{\hbar|x=0}=0$. %, $WF(v_{\hbar})\subset \{\tau>0\}$.
Recalling from Lemma \ref{lemorthog} that the eigenmodes are $e_{k}(x,\hbar^{-1})$ and using \eqref{eig_k}, we select a datum $v_{0}=v_{0}(x,a,1/\hbar)$ (to be suitably chosen later), decompose it over the eigenmodes and write the corresponding half-wave propagator, with an additional spectral cut-off $\chi_{0}(\omega_{k}) \chi_{1}(\omega_{k}\hbar^{2/3})$,
\begin{equation}
    \label{eq:46}
\begin{split}
  v_{\hbar}(t,x) & =        \begin{multlined}[t]                               \sum_{k\geq 1} e^{i\frac t \hbar (1+\omega_{k} \hbar^{\frac 23})^{\frac 1 2}}  \chi_{0}(\omega_{k} )\chi_{1}(\omega_{k}\hbar^{\frac 23}) e_{k} (x,\hbar^{-1}) \\  \times \int_{{z}>0} e_{k}({z},\hbar^{-1}) v_{0}({z},a,\hbar^{-1}) \,d{z}
\end{multlined}\\
& =      \begin{multlined}[t] \sum_{k\geq 1} \frac{2\pi \hbar^{-\frac 23} }{L'(\omega_{k})} e^{i\frac t \hbar (1+\omega_{k} \hbar^{\frac 23})^{\frac 1 2}} \chi_{0}(\omega_{k})\chi_{1}(\omega_{k}\hbar^{\frac 23}) Ai(\hbar^{-\frac 23} x-\omega_{k})\\ \times \int_{{z}>0} Ai(\hbar^{-\frac 23} {z}-\omega_{k}) v_{0}({z},a,\hbar^{-1}) \,d{z}\,.\end{multlined}
\end{split}
\end{equation}
It turns out to be convenient to localize $v_{\hbar}$ with respect to the Laplacian. Recall that $\sqrt{-\hbar^{2}\partial_{x}^{2}+(1+x)} e_{k}(x,\hbar^{-1})=\sqrt{1+\omega_{k} \hbar^{2/3}} e_{k}(x,\hbar^{-1})$), which explains why we added a spectral cut-off $\chi_{1}(\omega_{k}\hbar^{2/3})$, with $\chi_{1}(\zeta)=0$ for $\zeta<-1$ and $\zeta>2$, $\chi_{1}(\zeta)=1$ for $0<\zeta<1$. We also insert $\chi_{0}(\omega)=1$ for $\omega>2$, $\chi_{0}(\omega)=0$ for $\omega<1$: obviously $\chi_{0}(\omega_{k})=1$ for all $k$, as $\omega_{1}>2$. With both cut-offs, the sum over $k$ in \eqref{eq:46} is reduced to a finite sum $k \lesssim h^{-1}$, owning to the asymptotics of the zeroes of the Airy function, which are strictly positive and behave like $k^{2/3}$ for large $k$. Alternatively we may use the Green function formula \eqref{greenfct} and apply it to our datum $v_{0}$ (after inserting the same spectral cut-off in the Green function). We point out that our choice of $+$ sign in the half-wave propagator is arbitrary and does not play any important role beside setting a direction of propagation (to the left of the $x$ axis in the upper plane) when returning to $U(t,x,y)$. 

Using the Airy-Poisson formula \eqref{eq:AiryPoissonBis}, we transform the sum of eigenmodes (over $k$) into a sum over $N\in \Z$; its summands will be later seen to be waves corresponding to the number of reflections on the boundary, indexed by $N$ :
\begin{multline}
  v(t,x,\hbar^{-1})= \sum_{N\in \Z} \int_{\R}\int_{{z}>0} e^{-i NL(\omega)} \hbar^{-2/3} e^{i\frac t \hbar (1+\omega \hbar^{2/3})^{\frac 1 2}} \chi_{0}(\omega)\chi_{1}(\hbar^{2/3}\omega)\\
 Ai (\hbar^{-2/3} x-\omega) Ai(h^{-2/3}{z}-\omega) v_{0}({z},a,1/\hbar) \,d{z} d\omega\,.
\end{multline}
Recall that
\begin{equation}
  \label{eq:47}
  Ai(\hbar^{-2/3} x-\omega)=\frac 1 {2\pi \hbar^{1/3} } \int e^{\frac i \hbar (\frac{\sigma^{3}}{3}+\sigma(x-\hbar^{2/3}\omega))} \,d\sigma\,.
\end{equation}
If we rescale with $\zeta=\hbar^{2/3} \omega$, we get
\begin{equation}
  \label{eq:48}
    v(t,x,\hbar^{-1})= \frac 1 {(2\pi\hbar)^{2}} \sum_{N\in \Z} \int_{\R^{3}}\int_{{z}>0} e^{\frac i \hbar  \tilde \Phi_{N}}   \chi_{0}(\hbar^{-\frac 23}\zeta)\chi_{1}(\zeta) v_{0}({z},a,\hbar^{-1})  \, dz ds d\sigma d{\zeta}\,,
\end{equation}
where
\begin{equation}
  \label{eq:49}
    \tilde \Phi_{N}=\frac{\sigma^{3}} 3+\sigma(x-\zeta)+\frac {s^{3}} 3+s({z}-\zeta)-N\hbar L(\hbar^{-2/3} \zeta)+t \sqrt{1+\zeta}
\end{equation}
and therefore, with $\Phi_{N}=\tilde\Phi_{N}+y$, we find
\begin{equation*}
%  \label{eq:48bis}
  u(t,x,y)= \frac 1 {(2\pi)^{3}h} \sum_{N\in \Z} \int_{\R^{4}}\int_{{z}>0}e^{i \frac \eta h  \Phi_{N}}  \chi_{0}\Bigl(\frac {\zeta} {\hbar^{\frac 23}}\Bigr)\chi_{1}(\zeta) %\\{}  \times
  \frac {\eta^{2}} {h^{2}}{\psi(\eta)} v_{0}\Bigl({z},a,\frac \eta h\Bigr) \, d z ds d\sigma  d\zeta d\eta\,.
\end{equation*}
Let us rescale now like we did in \eqref{eq:38}, $t=a^{1/2} T$, $x=aX$, $y=-t\sqrt{1+a}+a^{3/2}Y$, with moreover
\begin{equation}
  \label{eq:50}
  \zeta=a E\,,\,\,s=a^{1/2}S\,,\,\, \sigma =a^{1/2} \Sigma\,,\,\, {z}=a {Z},
\end{equation}
then $u(t,x,y)$ becomes $U(T,X,Y)$ where, for $\lambda=a^{3/2}/h$ as before, we have
\begin{equation}
  \label{eq:51}  
\begin{split}
  U(T,X,Y)=  \frac {\lambda^{2}} {(2\pi)^{3}h} \sum_{N\in \Z} \int \int_{Z>0}  & e^{i\lambda \eta (Y+\Psi_{N})}V_{0}({Z},\lambda \eta) \\
   & {\quad {}\times
\chi_{\lambda}^{a}(E) \eta^{2}  \psi(\eta) \, dZ dS d\Sigma dE d\eta\,,}
\end{split}
\end{equation}
where $\chi_{\lambda}^{a}(E)=\chi_{0}(\eta^{2/3} \lambda^{2/3} E)\chi_{1}(a E)$. Here the phase function is given by 
\begin{equation}
  \label{eq:52}
  \Psi_{N}=\frac{\Sigma^{3}}3 +\Sigma(X-E)+\frac {S^{3}} 3+S({Z}-E)-\frac N{\lambda \eta} L((\lambda \eta)^{\frac 2 3} E)+\frac { T(E-1)}{\sqrt {1+a E}+\sqrt{1+a}}\,.
\end{equation}
The last term comes from the time propagator and takes into account the change of variable in $y$ that includes a time translation.

We conclude this introduction to the parametrix with an important lemma, in effect reducing the sum over $N$ in \eqref{eq:51} to a finite sum (with a very large number or terms).
\begin{lemma}\label{sumNfinie}
 Assume $V_{0}$ is a smooth function and $V_{0}(Z,\mu)\in L^{\infty}_{\mu}L^{1}_{Z}$. In the sum defining $U(T,X,Y)$ in \eqref{eq:51}, the only significant contributions arise from $N$'s such that $|N|\lesssim h^{-1/3}$.
\end{lemma}
\begin{proof}
We will rely on non stationary phase in either $E$, $S$ or $\Sigma$.  We have $\partial_{S}\Psi_{N}=S^{2}+Z-E$ and $\partial_{\Sigma}\Psi_{N}=\Sigma^{2}+X-E$. If either $|S|\geq 3 N^{1/100} E^{1/2} $ or $|\Sigma|\geq 3 N^{1/100}E^{1/2}$, integration by parts in one of these variables, say $S$, provides a factor $\lambda^{-1}|\partial_{S} \Psi|^{-1}\lesssim N^{-1/50}\lambda^{-1/3}$ using the lower bound on $E$ from its support. By non stationary phase, we get both enough decay to sum in $N$ and a bound  $\|V\|_{L^{1}_{Z}} O(\lambda^{-\infty})$ (the $(E,\eta)$ integral is bounded by support considerations and the $Z$ integral is bounded from $V_{0}\in L^{1}$). Using \eqref{eq:propL} to expand $L(\omega)$, 
\begin{multline*}
  \partial_{E}\Psi_{N}=\frac T {(\sqrt{1+aE}+\sqrt{1+a})}\frac{(1+aE+\sqrt{1+aE}\sqrt{1+a} +a/2)}{(1+aE +\sqrt{1+aE}\sqrt{1+a})} \\
  {}-S -\Sigma-2NE^{1/2}(1-\frac 3 4 B'(\lambda \eta E^{3/2}))
\end{multline*}
where the $B'$ term is small compared to $1$, if $\lambda E^{3/2}$ is sufficiently large ($>2$ is already enough). Note that the coefficient of $T$ is bounded from above and below by fixed constants, as $E>0$ and $aE\lesssim 1$. 
If $\sup(|S|,|\Sigma|)<3 N^{1/100}E^{1/2}$, then, for $|N|\geq 100$, $\Psi_{N}$ will not be stationary in $E$ provided that $|T|\lesssim (N-3N^{1/100}) E^{1/2} $ and non-stationary phase in $E$ provides, again, decay to sum in $N$ and an $O(\lambda^{-\infty})$ contribution. With the lower bound $E\gtrsim \lambda^{-2/3}$, the cardinal $\sharp N$ of the set of $N$'s that contribute is bounded by $|T| \lambda^{-1/3}$. As $T=t/a^{1/2}$, $\lambda=\frac{a^{3/2}}{h}$ and $a\gg h^{2/3}$, $\sharp N\lesssim \frac{h^{1/3}}{a}\ll h^{-1/3}$. Moreover, any $O(\lambda^{-\infty})$ is also an $O(h^{\infty})$. \end{proof}

\subsection{Choosing the initial data}
\label{sec:donnee}
We pick $v_0(z,a,\eta/h)$, which is now $V_0({Z},\lambda \eta)$, to be
\begin{equation}
  \label{eq:53}
  V_{0}({Z},\lambda \eta)=\int e^{i\lambda \eta(({Z}-1) s+\frac {s^{3}} 3 +\frac i 2 \frac {s^{2}} M)} \,ds\,.
\end{equation}
While we do not have $V_{0|{{Z}=0}}=0$, this will turn to be irrelevant for our purposes: the spectral cut-off $\chi_{1}$ insures that the datum $v_{\hbar}(0,x)$ is such that $v_{\hbar}(0,0)=0$ as a finite sum of Airy functions $Ai(-\omega_{k})$. Here $M$ is large and will be chosen later in this section, depending on $a,h$, while $\eta\sim 1$  through the $\psi(\eta)$ cut-off and therefore harmless. Defined in this way, $V_{0}$ is (microlocally) concentrated around $\{{Z}=1\}$ and the corresponding Fourier direction $\{ \Xi=0\}$:  we may explicitly compute $V_0$ as follows, with $\tilde \lambda=\lambda \eta$ and $\tau=\tilde\lambda^{1/3}s$ :
  \begin{align*}
\int e^{i\lambda \eta(({Z}-1) s+\frac {s^{3}} 3 +\frac i 2 \frac {s^{2}} M)} \,ds& =  \frac 1 {\tilde \lambda^{1/3}} \int e^{i(\frac {\tau^{3}} 3+ \tilde \lambda^{2/3}\tau({Z} -1)+ \frac i {2M} \tilde \lambda^{1/3} \tau^{2}) }\,d\tau\\
  &  =  \frac 1 {\tilde \lambda^{1/3}} \int e^{i\Big(\frac 13  (\tau+ \frac i {2M} \tilde \lambda^{1/3})^{ 3}+ \tilde \lambda^{2/3}\tau({Z} -1+\frac 1 {4M^{2}}) +\frac i 3 \tilde \lambda \frac 1 {(2M)^{3}}\Big)} \,d\tau\\
 & =  \frac 1 {\tilde \lambda^{1/3}}  e^{ \frac{\tilde \lambda}{2M} ({Z}-1+\frac 2 3 \frac 1 {4M^{2}})} 2\pi Ai\Big(\tilde \lambda^{2/3}({Z}-1+\frac 1 {4M^{2}})\Big)\,.
  \end{align*}
We select $1 \ll M\ll \lambda$: this will be our first condition on $M$. For $Z-1>1/10$, the exponential decay of $|Ai(z)|\sim \exp(-C z^{3/2})$ for large $z$  offsets the growth of the exponential factor in front of it, while for $Z-1<-1/10$, we get exponential decay in term of $\tilde \lambda /M$ from the front factor while $Ai$ is bounded. In particular, for ${Z}=0$, we get
\begin{equation}
  \label{eq:54}
\int e^{i\lambda \eta(({Z}-1) s+\frac {s^{3}} 3 +\frac i 2 \frac {s^{2}} M)} \,ds\Big|_{{Z}=0}=\frac 1 {\tilde \lambda^{1/3}}  e^{ -\frac{\tilde \lambda}{2M} (1-\frac 2 3 \frac 1 {4M^{2}})} 2\pi Ai\Big(\tilde \lambda^{2/3}(-1+\frac 1 {4M^{2}})\Big)\,,
\end{equation}
which is $O(\tilde \lambda^{-\infty})$ as $1-1/(6M^{2})>0$ and $\tilde \lambda /(2M) \geq \tilde \lambda^{\varepsilon}$ for a suitable (small) $\varepsilon>0$, provided that $1\ll M\ll \lambda$.  Moreover, with such a choice of $M$, we even have
\begin{equation}
  \label{eq:Zneg}
  V_{0|{Z}\leq 0}=\int e^{i\lambda \eta(({Z}-1) s+\frac {s^{3}} 3 +\frac i 2 \frac {s^{2}} M)} \,ds\Big|_{{Z}\leq 0}=O(\tilde \lambda^{-\infty})\,.
\end{equation}
We can then compute explicitely the Fourier transform of $V_0$, with  $1\ll M\ll \lambda $:
\begin{equation}
  \label{eq:55}
  \begin{split}
  \hat V_{0} (\tilde\lambda \xi,\tilde\lambda) & = \int e^{-i\tilde\lambda \xi {Z} }V_0({Z},\tilde\lambda)d{Z} \\
%  =\int_{\mathbb{R}}e^{-i\tilde\lambda \xi {Z} }\chi_2({Z}) \int e^{i\tilde\lambda (({Z}-1) s+\frac {s^{3}} 3 +\frac i 2 \frac {s^{2}} M)} \,dsd{Z}\\
 &   =\int \int e^{i\tilde\lambda (s-\xi){Z}} e^{i\tilde\lambda (\frac {s^{3}} 3 -s +\frac i 2 \frac {s^{2}} M)} \,dZ ds \\
  % =\int \frac{1}{\tilde\lambda}\delta(s-\xi)e^{i\tilde\lambda (\frac {s^{3}} 3 -s +\frac i 2 \frac {s^{2}} M)} \,ds
 & =\frac 1 {\tilde\lambda} e^{i\tilde\lambda (\frac{\xi^{3}}3-\xi+\frac {i \xi^{2}}{2M})}\,.
 \end{split}
\end{equation}
\subsection{$L^2$ norm of the initial data}
Define 
\[
u_{0}(x,y)=\frac 1h\int e^{i\frac{\eta}{h}y}v_0(x,a,\eta/h)\psi(\eta)d\eta,
\]
our initial data $u(0,x,y)$ will be the projection of $u_{0}$ over a finite number of spectral modes, through \eqref{eq:46}. By Bessel inequality, $\| u(0,\cdot)\|_{L^{2}(\Omega_{2})} \leq \| u_{0}\|_{L^{2}(\Omega_{2})}$, and using \eqref{eq:Zneg}, we have
\begin{equation}
  \label{eq:2}
 \|u_{0}\|_{L^{2}(\R^{2})}=\| u_{0}\|_{L^{2}(\Omega_{2})}+O(\lambda^{-\infty})\,.  
\end{equation}
 We therefore compute the Fourier transform of $u_{0}$, or its rescaled version: for $x=aX$, we set $v_0(x,a,\eta/h)=V_0(X,\lambda\eta)$ with $V_{0}$ defined by \eqref{eq:53}, and
\[
U_0(X,Y):=u(0,x,y)=\frac 1h \int e^{i\lambda \eta Y}V_0(X,\lambda \eta)\psi(\eta)d\eta.
\]
The Fourier transform of $U_0$ is obtained by a direct computation,
\begin{equation}\label{hatU0}\begin{split}
\hat{U_0}(\zeta,\lambda\eta) & =\int e^{-i\zeta X} e^{-i\lambda\eta Y}\frac 1h \int e^{i\lambda \tilde \eta Y}V_0(X,\lambda \tilde  \eta)\psi(\tilde\eta)d\tilde\eta dX dY\\
                           %  & =\frac 1h\int \int e^{i\lambda (\tilde\eta-\eta)Y}dY \psi(\tilde\eta) \hat{V_0}(\zeta,\lambda\tilde\eta)d\tilde\eta\nonumber\\
                 &  =\frac{1}{h\lambda}\int \delta_{\tilde\eta=\eta}\psi(\tilde\eta)\hat{V_0}(\zeta,\lambda\tilde\eta)d\tilde\eta\\
 & =\frac{1}{h\lambda}\psi(\eta)\hat{V_0}(\zeta,\lambda\eta).
\end{split}\end{equation}
We now estimate the $L^2$ norm of $U_0$, using the explicit form of $\hat V_{0}$ we already obtained:
\begin{equation}
    \label{eq:56}
  \begin{split}
  \|U_{0}\|^2_{L^{2}_{X\in\mathbb{R},Y}} & = \|\hat{U_{0}}\|^2_{L^{2}_{\zeta,\theta}}=\int |\hat{U_0}|^2(\zeta,\theta)d\zeta d\theta\\
% & =\lambda\int |\hat{U_0}|^2(\zeta,\lambda\eta)d\zeta d\eta\nonumber\\
 &  =\lambda \int \frac{1}{h^2\lambda^2}\psi^2(\eta)|\hat{V_0}|^2(\zeta,\lambda\eta)d\zeta d\eta\\
 &  =\frac{1}{h^2}\int \eta\psi^2(\eta)|\hat{V_0}|^2(\lambda\eta\xi,\lambda\eta)d\xi d\eta\\
                                  &   =\frac{1}{h^2\lambda^2}\int \eta^{-1}\psi^2(\eta)\Big|e^{i\eta\lambda(\xi^{3}/3-\xi+\frac i 2 \xi^{2}/M)}\Big|^2d\xi d\eta\\
&                      =  \frac{1}{h^2\lambda^2}\int \eta^{-1}\psi^2(\eta) e^{-\lambda \eta\xi^2/M} d\xi d\eta\\
&  \lesssim h^{-2}\lambda^{-5/2} M^{1/2}\,,
\end{split}
\end{equation}

where we used \eqref{hatU0} and \eqref{eq:55}. Recalling \eqref{eq:2}, this yields $ \|U_{0}\|_{L^{2}_{X\geq 0,Y}}\lesssim h^{-1}\lambda^{-5/4}M^{1/4}$.

\subsection{Computing the parametrix $U$} 
In the remainder of this section, we restrict ourselves to $a\gtrsim h^{1/2}$. For a suitable  chosen $M$, we prove Strichartz estimates \eqref{eq:41bis} to hold but with a loss in the parameter $\alpha$: $\alpha\leq 1/2-1/10$. We start by computing the $L^{\infty}$ norm of $U$, followed by its $L^q([0,1],L^{\infty})$ norm ; next, we balance lower bounds on space-time norms with our upper bound on the data, proving that if \eqref{eq:41bis} holds for $r=\infty$, this forces $q\geq 5$, which is equivalent to the aforementioned loss on $\alpha$. This provides our counterexample for the endpoint Strichartz estimate $(q,+\infty)$. We then compute the $L^{r}$ norm of $U$ to recover other exponents, and this is ultimately useful in higher dimensions as well.

The phases $\Psi_N$ in the sum defining $U$ (in \eqref{eq:51}) are all linear in ${Z}$: we replace $V_0$ given by \eqref{eq:53} in \eqref{eq:51} and, using Lemma \ref{sumNfinie}, we restrict (up to an $O(h^{\infty})$ term) to a finite sum over $|N|\lesssim h^{-1/3}$ (note that $V_{0}\in L^{1}_{Z}$ from the previous pointwise bounds we obtained). In this finite sum, we may add the same integrals but with $Z<0$: these add up to the $O(h^{\infty})$ term, as $V_{0}$ is asymptotically small for $Z<0$. The inner integral over ${Z}\in \R$ yields
$$
\int e^{i\lambda \eta {Z}(s+S)}\,d{Z}= \frac{2\pi}{\lambda \eta} \delta(s+S)\,,
$$
therefore we get
\begin{equation}
  \label{eq:57}
  U(T,X,Y)=\frac{\lambda}{(2\pi)^{2}h}  \sum_{|N|\lesssim h^{-1/3}} \int e^{i\lambda \eta(Y+\varphi_{N})} \chi_{\lambda}^{a}(E) \psi(\eta) \eta \,d\eta ds d\Sigma dE+O({h^{-\infty}})
\end{equation}
where $\varphi_{N}$ is the (complex) phase
\begin{equation}
  \label{eq:58}
 \varphi_{N}= T \frac{(E-1)}{ \sqrt{1+aE}+\sqrt{1+a}} - \frac N {\lambda \eta} L((\lambda \eta)^{2/3} E)+s (E-1) + i \frac {s^{2}}{2M} +\frac{\Sigma^{3}} 3+ \Sigma (X-E)\,.
\end{equation}
We start by eliminating the $s$ variable in the integral from \eqref{eq:57} with complex phase function $\varphi_N$ defined in \eqref{eq:58}. We have
\begin{equation}
  \label{eq:61}
  \int e^{i\lambda \eta  (s(E-1)+i\frac {s^{2}}{2M})} 
\,ds= \sqrt{\frac{2\pi}{\lambda\eta}} \sqrt M e^{-\frac{\lambda\eta M(E-1)^{2}}{2}} 
\end{equation}
and therefore \eqref{eq:57} becomes
\begin{equation}
  \label{eq:62}
  U(T,X,Y)=\frac{\lambda^{\frac 12}M^{\frac 12}}{ (2\pi)^{\frac 32}h}\sum_{|N|\lesssim h^{-\frac 13}}\int e^{i\lambda \eta(Y+\tilde\varphi_{N})} \chi_{\lambda}^{a }(E)\psi(\eta)\eta^{\frac 12} \,d\Sigma dE d\eta+O(h^{\infty})
\end{equation}
with phase 
\begin{equation}
 \label{eq:63}
\tilde\varphi_{N}   = T \frac{(E-1)}{ \sqrt{1+aE}+\sqrt{1+a}} - \frac N {\lambda \eta} L((\lambda \eta)^{2/3} E)\\{}+ i \frac {M(E-1)^{2}}{2} +\frac{\Sigma^{3}} 3+ \Sigma (X-E)\,.
  \end{equation} 
Recall that $L(\omega)=\frac\pi 2 + \frac 43 \omega^{3/2}-B(\omega^{3/2})$ where $B(u)\sim_{1/u}\sum_{n\geq 1} b_{n}u^{-n}$, hence 
\[
\frac N {\lambda \eta} (L((\lambda \eta)^{2/3} E)-\frac \pi 2 )=\frac N {\lambda \eta} \Big(\frac 43 \lambda\eta E^{3/2}-B(\lambda\eta E^{3/2})\Big)=\frac 43 NE^{3/2}-\frac N {\lambda \eta}B(\lambda\eta E^{3/2}).
\]

\begin{rmq}\label{rmqBNB}
For values $h^{1/2}\lesssim a$, the factor $\exp{(iNB)}$ in our phase does not oscillate anymore: indeed, the phase $\varphi_N$ given in \eqref{eq:58} is stationary in $E$ only when $N\sim T\sim \frac{t}{\sqrt{a}}$ and for $E$ near $1$ (which is forced by the imaginary part of the phase)
\[
NB(\lambda\eta E^{3/2})\sim \frac{N}{\lambda}\sim \frac{t}{\sqrt{a}}\times \frac{h}{a^{3/2}}\lesssim \frac{h}{a^2}.
\]
Therefore, when $h^{1/2} \ll a$ we can actually bring the $\exp{iNB(\cdot)}$ factor in the symbol rather than leave $NB(\cdot)$ in the phase (in order to do explicit computations).
\end{rmq}
We have, from \eqref{eq:63},
\begin{gather}
  \mathrm{Im} (\tilde \varphi_{N})= \frac{M(E-1)^{2}}{2}\\
\partial_{\Sigma} \tilde \varphi_{N}=\Sigma^{2}+X-E\\
\begin{multlined}
  \partial_{E} \tilde \varphi_{N}= T\partial_{E}\left(   \frac{(E-1)}{ \sqrt{1+aE}+\sqrt{1+a}}\right) - \frac N {(\lambda \eta)^{1/3}} L'((\lambda \eta)^{2/3} E)\\
  {}+i\frac M 2 (E-1)-\Sigma\,.
  \end{multlined}
\end{gather}
Therefore, the set $\{\mathrm{Im}(\tilde \varphi_{N})=0,\,\, \nabla_{(\Sigma,E)}\tilde \varphi_{N}=0\}$ coincides with\begin{equation}
  \label{eq:59}
  \{  E=1\,,\,\, \Sigma=\frac{T}{2\sqrt {1+a}} -  \frac N {(\lambda \eta)^{1/3}} L'((\lambda \eta)^{2/3} )\,,\,\,X=1-\Sigma^{2}\}.
\end{equation}
In the $(T,X)$ plane, this is the trajectory moving to the right from $X=1$, $\Sigma=0$. We introduce the following notations : let $\varepsilon_m>0$, $m\in\{0,1,2\}$ be small, $J\in \Z$ and set
\[
I_{J}=4J \sqrt{1+a}+(-2\varepsilon_{0},2\varepsilon_{0}),
\]
\[
\DomainXYT_{J}=\{T\in I_{J}\,,\,\, |X-1|\leq \varepsilon_{1}\,,\,\, |Y-4 J/3|\leq \varepsilon_{2}\,\}.
\]
From now on we will focus on $U$ restricted to a set $\DomainXYT_J$ on which we obtain a lower bound of its $L^{\infty}$ norm. We first need the following result, which states that, if for a given $J$ we consider only points $(T,X,Y)\in \DomainXYT_J$, then in the sum \eqref{eq:62} defining $U(T,X,Y)$ indexed over the number of reflections $N$ there is only one single integral that provides a non-trivial contribution, corresponding to $N=J$.
\begin{prop}
\label{propCE}
For all $n \in \N^{*}$, there exists $C_{n}$ such that for all $0\leq J\lesssim M_a$, for all $1\ll M\ll \lambda$ and for all $(T,X,Y)\in \DomainXYT_{J}$, the following holds
\begin{equation}
  \label{eq:60}
  \left|{\, U(T,X,Y)-\frac{\lambda^{1/2}M^{1/2}}{ (2\pi)^{3/2}h}\int e^{i\lambda \eta(Y+\tilde\varphi_{J})}\chi_{\lambda}^{a}(E)\psi(\eta)\eta^{1/2} \,d\Sigma dE d\eta \, }\right|
%  \frac \lambda 1 {(2\pi)^{2}h}  \int e^{i\lambda \eta(Y+\varphi_{J})} \chi_{1}(aE) \eta \psi(\eta)\, d\eta d\Sigma d s d E
 \leq C_{n} \lambda^{-n}\,.
\end{equation}
\end{prop}
\begin{proof}

Let $0\leq J\lesssim M_a$ and let $(T,X,Y)\in\DomainXYT_J$: we can write 
\[
T=(4J+2\tilde T)\sqrt{1+a}, \quad X=1+\tilde X,\text{ where } |\tilde T|\leq \varepsilon_0, \quad |\tilde X|\leq \varepsilon_1.
\] 
We also change variable $E=1+(1+a){\tilde E}$: from the Gaussian nature of $\tilde\varphi_N$, $E$ has to stay close to $1$. Using Remark \ref{rmqBNB}, we may move $\exp{iNB(\cdot)}$ in the symbol and relabel the phase to remove the harmless factor $i N \pi/2$; with new variables $\tilde X$ and ${\tilde E}$, the relabelled $\tilde\varphi_N$ reads
\begin{equation}
  \label{eq:65}
  \begin{split}
    \tilde\varphi_{N} =  & \frac{T\tilde E \sqrt{1+a}}{ 1+\sqrt{1+a{\tilde E}}} - \frac 4 3 N (1+(1+a){\tilde E})^{3/2} \\
     & {\quad{}+\frac{\Sigma^{3}} 3+ \Sigma (\tilde X-(1+a){\tilde E})+\frac i 2 M (1+a)^2 {\tilde E}^{2}\,.}
\end{split}
\end{equation}
The derivatives with respect to ${\tilde E},\Sigma$ are
\begin{align*}
  \partial_{\Sigma}\tilde\varphi_N & =\Sigma^2+\tilde X-(1+a){\tilde E}\,\\  
  \partial_{{\tilde E}}\tilde\varphi_N & = \begin{multlined}[t]
T\sqrt{1+a}\Big(\frac{1}{1+\sqrt{1+a{\tilde E}}}-\frac{a{\tilde E}}{2\sqrt{1+a{\tilde E}}(1+\sqrt{1+a{\tilde E}})^2}\Big)\\{}-2N(1+a)(1+(1+a){\tilde E})^{1/2}-(1+a)\Sigma+iM(1+a)^2{\tilde E}\,.
\end{multlined}
\end{align*}
Obviously the set $\{ \mathrm{Im}(\tilde\varphi_{N})=0\,,\,\, \nabla_{({\tilde E},\Sigma)} \tilde\varphi_{N}=0\}$ is given by
\begin{equation}
  \label{eq:66}
\Big\{\,  {\tilde E}=0\,,\,\, \tilde X+\Sigma^{2}=(1+a){\tilde E}=0\,,\,\, \Sigma=\Big(\frac T {2(\sqrt{1+a})} -2N\Big)\,\Big\},
\end{equation}
and therefore, imposing $|\tilde X|\leq \varepsilon_{1}$ implies $|\Sigma|\leq \varepsilon_{1}^{1/2}$ which yields
\[
|\frac T{2\sqrt{1+a}}-2N|=|2J+\tilde T-2N|\leq \varepsilon_{1}^{1/2}.
\] 
From $|\tilde T|\leq \varepsilon_{0}$, we get that for $\varepsilon_{0,1}<\frac 14$, the last inequality forces $N=J$. This proves Proposition \ref{propCE} as for $N\neq J$, we can perform non stationary phase, gaining powers of $\lambda$, and the sum is finite with at most $h^{-1/3}$ terms.% as well as powers of $N$ through $\partial_{\tilde E}\tilde \varphi_{N}$ (to insure summability in $N$).
\end{proof}

Using Proposition \ref{propCE} and Remark \ref{rmqBNB} we may rewrite, for $(T,X,Y)\in \DomainXYT_{J}$, recalling that $O(\lambda^{-\infty})=O(h^{\infty})$ as $a\gtrsim  h^{1/2}$,
\begin{multline}
  \label{eq:67}
  U((4J+2\tilde T)\sqrt{1+a},1+\tilde X,Y) = \frac{(1+a)\sqrt {\lambda M}}{(2\pi)^{\frac 32}h} (-i)^{J} \\
\times \int e^{i\lambda\eta (Y+\tilde\psi_{M}+JF)} \sigma_{J,\lambda }(\tilde E,\eta)\, d\Sigma d {\tilde E} d\eta +O(h^{\infty})\,,
  \end{multline}
where $\tilde\varphi_J$ was replaced by $\tilde\psi_{M}(\cdot)+JF({\tilde E})$: in the new variables, $\tilde\psi_{M}$ and $F$ are respectively
\begin{align}
\tilde\psi_{M} (\tilde T,{\tilde E},\Sigma) & =\frac{2\tilde T {\tilde E}(1+a)}{1+\sqrt{1+a{\tilde E}}}+i\frac M2 (1+a)^2 {\tilde E}^2+\frac{\Sigma^3}{3}+\Sigma(\tilde X-(1+a){\tilde E})\,, \\
  \label{eq:68}
F({\tilde E}) & =\frac{4{\tilde E}(1+a)}{1+\sqrt{1+a{\tilde E}}}-\frac 43 (1+(1+a){\tilde E})^{3/2}\,,
\end{align}
and the symbol is $\sigma_{J,\lambda}(E(\tilde E),\eta)$ with  $E(\tilde E)=1+(1+a)\tilde E$ and $$\sigma_{J,\lambda}(E,\eta)=\chi^a_{\lambda}(E)\eta^{1/2}\psi(\eta) \exp(i NB(\lambda \eta E^{3/2}))\,.
$$
Since $\mathrm{Im}(\tilde\psi_{M})=0$ only at ${\tilde E}=0$, we expand $F$ near ${\tilde E}=0$,
\begin{gather*}
  F(0)=-\frac 43, \quad F'(0)=0, \quad F''(0)=-(1+a)(1+2a)\,,\\
F({\tilde E})= -\frac 43-\frac{{\tilde E}^2}{2}(1+a)(1+2a)+O({\tilde E}^3)\,.  
\end{gather*}
Our new phase function $\tilde\psi_{M}+JF$ depends on two large parameters : $M$, to be chosen such that $1\ll M\ll\lambda $ and $J$, taking all values from $1$ to $M_a=a^{-1/2}$, depending on the region $\DomainXYT_J$ containing $(T,X,Y)$.

Let us take $J\leq  M_a\lesssim M$ : in the phase $\lambda\eta(\tilde\psi_{M}+JF)$, we consider the large parameter to be $M\lambda$ and, for $\Lambda= M \lambda (1+a)$, we get
\begin{multline}\label{eq:phaUj}
\lambda\eta(\tilde\psi_{M}+JF)=\lambda \eta \Big(\frac{\Sigma^3}{3}+\Sigma \tilde X-\frac 43 J\Big)+
\Lambda \eta\Big[\Big(\frac{2\tilde T}{1+\sqrt{1+a{\tilde E}}}-\Sigma\Big)\frac {\tilde E}M\\
+\frac i2 (1+a){\tilde E}^2-\frac{J}{2M}{\tilde E}^2(1+2a)+O(\frac{J}{M}{\tilde E}^3)\Big].
\end{multline}
\begin{rmq}\label{rmkimagphase}
In the integral \eqref{eq:67}, we may localize on $|{\tilde E}|\lesssim\Lambda^{-1/2}$ using the imaginary part of the phase; indeed, for larger values of ${\tilde E}$ the phase is exponentially decreasing ; we can then localize near the critical points in $\Sigma$, and $\Sigma^{2}=(1+a){\tilde E}-X$ hence $\Sigma$ becomes uniformly bounded and $\frac 1M\Big|\frac{2\tilde T}{1+\sqrt{1+a{\tilde E}}}-\Sigma\Big | \in O(1/M)$. Moreover, for $J\leq M_a\lesssim M$, the imaginary phase factor $\exp{(i\Lambda\eta(1+a)O(\frac JM {\tilde E}^3)})$ does not oscillate for values $|{\tilde E}|\lesssim \Lambda^{-1/2}$ (i.e. for ${\tilde E}$ such that the contribution of the integral is not exponentially small).
\end{rmq}
\begin{rmq}
Writing, for small ${\tilde E}$, $\frac{2\tilde T}{1+\sqrt{1+a{\tilde E}}}= \tilde T(1-\frac a4 {\tilde E}+O(a^2{\tilde E}^2))$, we obtain the first few terms of the Taylor expansion in $\tilde E$ of the phase with large parameter $\Lambda \eta$ as follows
\[
(\tilde T-\Sigma)\frac {\tilde E}M+\frac i2 \nu_a {\tilde E}^2+O(\frac JM {\tilde E}^3),\quad \nu_a=1+a+i \Big(\frac JM(1+2a)+\frac{a \tilde T}{2M}\Big).
\] 
\end{rmq}
\begin{rmq}
We are still carrying a symbol $\sigma_{J,\lambda}$; we may safely discard its $\chi_{\lambda}^{a}(E)$ component as $E$ is now localized near $E=1$, and therefore the contributions coming from $(1-\chi_{1}(aE))$ and $(1-\chi_0((\lambda\eta)^{2/3}E))$ are harmless by non stationary phase, and the remaining $\tilde \sigma_{J,\lambda}(E,\eta)=\exp(i J B(\lambda \eta E^{3/2}))$ is elliptic, close to $1$ near $E=1$ and $J/\lambda \ll 1$.
\end{rmq}
We rewrite the integral in ${\tilde E},\Sigma$ in \eqref{eq:67} as
\begin{multline*}
  \int e^{i \lambda\eta (\tilde\psi_{M}+JF)} \sigma_{J,\lambda}\,  d {\tilde E} d\Sigma =  \int e^{i \lambda\eta (\frac{\Sigma^3}{3}+\Sigma \tilde X-\frac 43 J)}  e^{i \Lambda \eta  \left((\tilde T-\Sigma) \frac {\tilde E} M + \frac i2 \nu_a {\tilde E}^{2}+O(\frac J M {\tilde E}^{3})\right)} \\
  \times \tilde \sigma_{J,\lambda }\,  d {\tilde E} d\Sigma +O(h^{\infty}) \,
\end{multline*}
and apply stationary phase in ${\tilde E}$ with complex phase $(\tilde T-\Sigma) \frac {\tilde E} M + \frac i2 \nu_a {\tilde E}^{2}+O(\frac J M {\tilde E}^{3})$ and large parameter $\Lambda\eta$. The second derivative's absolute value equals $|1+i\frac JM+O(a)|+O(\frac JM {\tilde E})\sim \sqrt{1+\frac{J^2}{M^2}}+O(\frac JM\times \Lambda^{-1/2})\sim 1$ for $J\leq M_a\lesssim M$, and stationary phase yields
\begin{equation}
  \label{eq:70}
  \begin{split}
    \int e^{i \lambda\eta (\tilde\psi_{M}+JF)}\sigma_{J,\lambda } \,  d {\tilde E} d\Sigma  = &  \int e^{i \lambda\eta (\frac{\Sigma^3}{3}+\Sigma \tilde X-\frac 43 J)} \\
   & {\quad} \times \sqrt{\frac{2\pi}{\nu_a \Lambda\eta}} e^{ \Lambda\eta  \frac{\nu_a}{2}( {\tilde E}^{2}_{c}+O({\tilde E}_{c}^{3}))}\sigma^{c}_{J,\lambda}\,d\Sigma+O(h^{\infty})  \end{split}
\end{equation}
where the critical point is ${\tilde E}_{c}=\tilde E_c(\Sigma)=i(\tilde T-\Sigma)/(M\nu_a) (1+O(|\tilde T-\Sigma|/M))$ and $\sigma^{c}_{J,\lambda}(\Sigma,\eta)$ is an elliptic symbol with an asymptotic expansion over $\Lambda^{-1}$, with leading order contribution $\tilde \sigma_{J,\lambda}(E(\tilde E_{c}(\Sigma)),\eta)$. 

\begin{rmq}\label{rmqsymbC}
Using Remark \ref{rmkimagphase}, if $|\tilde E_c(\Sigma)|\geq \Lambda^{-(1-\epsilon)/2}$ for some $\epsilon>0$, then the integral in the right hand side term in \eqref{eq:70} is exponentially small. On the other hand, for $|\tilde E_c|\ll 1$, $E(\tilde E_c)=1+(1+a)\tilde E_c$ stays close to $1$ hence $\chi^a_{\lambda}(E(\tilde E_c))$ vanishes together with all its derivatives. Therefore, for all $\Sigma$ such that $\frac{|\tilde T-\Sigma|}{M|\nu_a|}\ll 1$ we have
\begin{equation}\label{symbsigmaCJlambda}
\sigma_{C,J,\lambda}(\Sigma,\eta)=e^{iJB(\lambda\eta E(\tilde E_c(\Sigma))^{3/2})}\Big(1+O(\frac{J}{\Lambda\lambda})\Big)\psi(\eta).
\end{equation}
In particular, \eqref{symbsigmaCJlambda} holds for $|\tilde E_c|\lesssim  \Lambda^{-(1-\epsilon)/2}$, i.e. there where the integral in the RHS term of \eqref{eq:70} is not exponentially decreasing.
\end{rmq}

Since $\Lambda=\lambda M(1+a)$, we have, at $T=(4J+2\tilde T)\sqrt{1+a}$, $X=1+\tilde X$,
\begin{equation}
  \label{eq:71}
|U(T,X,Y)| \sim \frac 1h \left| \int e^{i\lambda\eta (Y-\frac 43J+\frac i 2 M (1+a) \nu_a P^{2}(1+O(P)) +\frac{\Sigma^{3}}{3}+\Sigma\tilde X)}\sigma^{c}_{J,\lambda}(\Sigma,\eta)\, d\Sigma d \eta \right|%\,,
\end{equation}
with $P=(\tilde T-\Sigma)/(M\nu_a)$. We are now left with the $\Sigma$ integration:
\begin{equation}
  \label{eq:72}
  I(\tilde T,\tilde X,\eta)=\int e^{i\lambda \eta G(\Sigma,\tilde T,\tilde X)}\sigma^{c}_{J,\lambda}(\Sigma,\eta)\,d\Sigma,
\end{equation}
where $\sigma^{c}_{J,\lambda}$ is elliptic, $\sigma^{c}_{J,\lambda}(\Sigma,\eta)=e^{iJB(\lambda\eta E(\tilde E_c)^{3/2})}(1+O(\frac{J}{\Lambda\lambda}))\psi(\eta)$ for $|P|=\frac{|\tilde T-\Sigma|}{M|\nu_a|}\ll 1$ and
\begin{equation}
  \label{eq:73}
  G(\Sigma, \tilde T, \tilde X)=\frac i 2 \frac{(1+a)}{M\nu_a} (\tilde T-\Sigma)^{2}\Big(1+O(\frac{\tilde T-\Sigma}{M\nu_a})\Big) +\frac{\Sigma^{3}}{3}+\Sigma\tilde X.
\end{equation}
We first discard the $O(P)$ term, as it may be seen later as an harmless perturbation, and forget about the symbol for now: consider
\begin{equation}
  \label{eq:72bis}
  I_{0}(\tilde T,\tilde X,\eta)=\psi(\eta)\int e^{i\lambda \eta G_{0}(\Sigma,\tilde T,\tilde X)}\,d\Sigma,
\end{equation}
with phase
\begin{equation}
  \label{eq:73bis}
  G_{0}(\Sigma,\tilde T, \tilde X)=\frac{i(1+a)}{2 M \nu_a} (\tilde T-\Sigma)^{2}+\frac{\Sigma^{3}}{3}+\Sigma\tilde X=\gamma (\tilde T-\Sigma)^2+\frac{\Sigma^3}{3}+\Sigma \tilde X,
\end{equation}
where we have set $\gamma := \frac i2\frac{(1+a)}{M\nu_a}$, e.g. $\gamma=\frac{i}{2(M+i J+ i\frac{a}{(1+a)}(J+\tilde T/2))}$.
We are after lower bounds for the $L^{\infty}$ norm of $U$, hence we seek values of $\tilde X$ where $I_0$ reaches its maximum. Write
\[
  G_{0}(\Sigma,\tilde T, \tilde X)=\gamma(\gamma^2+2\gamma \tilde T-\tilde X)+\gamma \tilde T^2-\gamma^3/3+ (\Sigma+\gamma)^3/3-(\Sigma+\gamma)(\gamma^2+2\gamma \tilde T-\tilde X).
\]
The last two terms (the only ones depending on $\Sigma$) may be seen as an Airy phase function, and therefore we have
\begin{equation}\label{Airybound}
  \begin{split}
    I_0(\tilde T,\tilde X,\eta)  = & \psi(\eta)e^{i\tilde \lambda (\gamma(\gamma^2+2\gamma \tilde T-\tilde X)+\gamma \tilde T^2-\frac{\gamma^3}3)}\int e^{i\tilde \lambda ( \frac{(\Sigma+\gamma)^3}3-(\Sigma+\gamma)(\gamma^2+2\gamma \tilde T-\tilde X))} d\Sigma\\
 = & \psi(\eta)\tilde\lambda^{-\frac 13}e^{i\tilde \lambda (\gamma(\gamma^2+2\gamma \tilde T-\tilde X)+\gamma \tilde T^2-\frac {\gamma^3}3)}Ai(-\tilde\lambda^{2/3}(\gamma^2+2\gamma \tilde T-\tilde X))\,.
  \end{split}
\end{equation}
Recall $Ai(0)=\frac{1}{3^{2/3}\Gamma(2/3)}\geq 3/10$; moreover there exists a small constant $0<c <1$ such that $|Ai(z)|>1/10$ for all $z\in\mathbb{C}$ with $|z|\leq c$. We can therefore bound from below the modulus of the Airy function in \eqref{Airybound} as follows
\begin{equation}\label{boundAiryTX}
|Ai(-\tilde\lambda^{2/3}(\gamma^2+2\gamma \tilde T-\tilde X))|>\frac 1 {10}
\end{equation}
for all $|\tilde\lambda^{2/3}(\gamma^2+2\gamma \tilde T-\tilde X)|\leq c$. Here $\tilde T,\tilde X$ are real, while $\gamma$ takes complex values and satisfies $|\gamma|=|\frac{i}{2(M+iJ+O(a))}|\sim \frac{1}{M}$ for $J\leq M_a\lesssim M$. Taking $M\geq \frac{4\lambda^{1/3}}{c}$ it follows that \eqref{boundAiryTX} holds true for all $\tilde T\leq  \frac{1}{\lambda^{1/3}}$ and all $|\tilde X|\leq \frac{c}{2\lambda^{2/3}}$.
We now study the behavior of the exponential factor in \eqref{Airybound}. For $\tilde T,\tilde X$ such that \eqref{boundAiryTX} holds we have
\[
 \lambda |\gamma(\gamma^2+2\gamma \tilde T-\tilde X)|\leq c\lambda^{1/3}|\gamma|\leq 1/4.
\]
For $|\gamma|\sim \frac{1}{M}\leq \frac{c}{4\lambda^{1/3}}$ and $\tilde T\leq \frac{1}{\lambda^{1/3}}$, the remaining term in the exponential factor of the Airy integral in \eqref{Airybound} can be bounded as follows
\begin{equation}\label{condT}
\lambda |\gamma \tilde T^2-\gamma^3/3|\leq \lambda^{1/3}|\gamma|+\lambda |\gamma|^3/3\leq c/3\leq 1/3.
\end{equation}
\begin{rmq}
The condition $\lambda^{1/3}\lesssim M$ must hold in order for \eqref{boundAiryTX} to hold and the term $|\tilde\lambda\gamma^3|$ in the exponential factor to stay bounded. For such $M$, to get $|2\gamma \tilde T-\tilde X|\lesssim \lambda^{-2/3}$ we require $|\tilde T|\lesssim \frac{M}{\lambda^{2/3}}$ and $|\tilde X |\lesssim \lambda^{-2/3}$. On the other hand, the condition $\lambda |\gamma|\tilde T^2\lesssim 1$ gives $|\tilde T|\lesssim \sqrt{M/\lambda}$; as $M/\lambda^{2/3}\geq \sqrt{M/\lambda}$ for all $M>\lambda^{1/3}$, in order to have $|I_0(\tilde T,\tilde X)|\sim \lambda^{-1/3}$ we must ask \begin{equation}\label{condTX}
\lambda^{1/3}\lesssim M,\quad |\tilde T|\lesssim  \sqrt{M/\lambda}, \quad |\tilde X|\lesssim \lambda^{-2/3}.
\end{equation}
In particular, taking $M\sim \lambda^{1/3}$ yields $|\tilde T|\lesssim \lambda^{-1/3}$.
\end{rmq}
Let us assume for a moment that the part of phase function depending on $\Sigma$ in \eqref{eq:71} is $G_0$ (instead of $G$) and there is no symbol: \[
|U((4J+2\tilde T)\sqrt{1+a},1+\tilde X,Y)| \sim \frac 1h \left| \int e^{i\lambda\eta (Y-\frac 43J+G_0(\Sigma,\tilde T,\tilde X))}\psi(\eta) \, d\Sigma d \eta \right|\,.
\]
Then, using \eqref{Airybound}, we would immediately get
\begin{multline}\label{U}
|U((4J+2\tilde T)\sqrt{1+a},1+\tilde X,Y)|\sim \frac{\lambda^{-1/3}}{h}\Big|\int e^{i\lambda\eta (Y-\frac 43J+\gamma \tilde T^2-\frac{\gamma^3}{3})}\\\times e^{i\eta \lambda \gamma(\gamma^2+2\gamma \tilde T-\tilde X)}Ai(-(\eta\lambda)^{2/3}(\gamma^2+2\gamma \tilde T-\tilde X))\eta^{-1/3}\psi(\eta)d\eta\Big|,
\end{multline}
and from the discussion above, for $\tilde T,\tilde X$ and $M$ like in \eqref{condTX}, the factor from the second line in \eqref{U}, 
$e^{i\eta \lambda \gamma(\gamma^2+2\gamma \tilde T-\tilde X)}Ai(-(\eta\lambda)^{2/3}(\gamma^2+2\gamma \tilde T-\tilde X))$, may be seen as part of the symbol (it does not oscillate). With \eqref{condT} holding,  we move $e^{i\eta \lambda(\gamma \tilde T^2-\gamma^3/3)}$ into the symbol as well.
The (remaining) phase in \eqref{U} is $\eta \lambda (Y-\frac 43J)$ and therefore $Y$ takes values in a ball of center $\frac 43J$ and radius $\lambda^{-1}$.
%\begin{rmq}

We now fix the heuristic assuming that $G$ could be replaced by $G_0$ and there is no symbol; note that we obtained an additional condition on $M$, that is $\lambda^{1/3}\lesssim M$. In the next lemma we prove the difference between $I(\tilde T,\tilde X,\eta)$ in \eqref{eq:72} and $e^{iJB(\lambda\eta)} I_0(\tilde T,\tilde X,\eta)$ with $I_0(\tilde T,\tilde X,\eta)$ given in \eqref{eq:72bis} to be lower order:
\begin{lemma} 
The following holds
\[
I(\tilde T,\tilde X,\eta)=e^{iJB(\lambda\eta)}I_0(\tilde T,\tilde X,\eta)+O(M^{2\epsilon}/\lambda^{-(1-2\epsilon)}).
\]
\end{lemma}
\begin{rmq}
In the next section we will take $M\sim \lambda^{1/3}$, and later $M\lesssim \lambda$, both of which produce a lower order remainder for all $0<\epsilon<1/6$.
\end{rmq}
\begin{proof}
Let us definie $\mathcal{C}(\Sigma)=|\tilde T-\Sigma|/(M|\nu_a|)\leq \Lambda^{-(1-\epsilon)/2}$, with small $\epsilon>0$. From \eqref{eq:72} and Remark \ref{rmqsymbC},
\[
I(\tilde T,\tilde X,\eta)=\int_{\mathcal{C}(\Sigma)}e^{i\lambda \eta G(\Sigma,\tilde T,\tilde X)}\sigma^{c}_{J,\lambda}(\sigma,\eta)d\Sigma+O(\lambda^{-\infty}).
\]
Replacing $\sigma^{c}_{J,\lambda}$ in the last integral by \eqref{symbsigmaCJlambda} yields
\[
I(\tilde T,\tilde X,\eta)=\psi(\eta)\int_{\mathcal{C}(\Sigma)}e^{i\lambda \eta G(\Sigma,\tilde T,\tilde X)}e^{iJB(\lambda\eta E(\tilde E_c(\Sigma))^{3/2})}\Big(1+O(\frac{J}{\Lambda\lambda})\Big)d\Sigma+O(\lambda^{-\infty}).
\]
The last integral can be re-arranged as follows
\begin{equation}\label{equaint}
  \begin{multlined}
    \psi(\eta)e^{iJB(\lambda\eta)}\int_{\mathcal{C}(\Sigma)} e^{i\lambda \eta G_0(\Sigma,\tilde T,\tilde X)}    e^{i\lambda \eta (G-G_0)(\Sigma,\tilde T,\tilde X)}
    \\\times e^{iJ(B(\lambda\eta E(\tilde E_c(\Sigma))^{3/2})-B(\lambda\eta))}\Big(1+O(\frac{J}{\Lambda\lambda})\Big)d\Sigma\\
    =\psi(\eta)e^{iJB(\lambda\eta)}\int_{\mathcal{C}(\Sigma)}e^{i\lambda \eta G_0(\Sigma,\tilde T,\tilde X)}\\
    \times \Big[e^{i\lambda \eta (G-G_0)(\Sigma,\tilde T,\tilde X)}e^{iJ(B(\lambda\eta E(\tilde E_c(\Sigma))^{3/2})-B(\lambda\eta))}\Big(1+O(\frac{J}{\Lambda\lambda})\Big)-1\Big]d\Sigma
\\
+\psi(\eta)e^{iJB(\lambda\eta)}\int_{\mathcal{C}(\Sigma)}e^{i\lambda \eta G_0(\Sigma,\tilde T,\tilde X)}d\Sigma\,.
  \end{multlined}
\end{equation}
As $I_0(\tilde T,\tilde X,\eta)=\psi(\eta)\int_{\mathcal{C}(\Sigma)}e^{i\lambda \eta G_0(\Sigma,\tilde T,\tilde X)}d\Sigma+O(\lambda^{-\infty})$, the third line of \eqref{equaint} is nothing but $e^{iJB(\lambda\eta)}I_0(\tilde T,\tilde X,\eta)+O(\lambda^{-\infty})$. It remains to evaluate the integral in the second line of \eqref{equaint}. Using \eqref{eq:73} and \eqref{eq:73bis}, we have, as $|\nu_a|\sim 1$, $\Lambda=M\lambda$,
\[
\lambda \eta |G-G_0|(\Sigma,\tilde T,\tilde X)=\lambda \eta\frac{(1+a)}{2}\frac{(\tilde T-\Sigma)^2}{M|\nu_a|}O\Big(\frac{|\tilde T-\Sigma|}{M|\nu_a|}\Big)\sim \Lambda\Big( \frac{|\tilde T-\Sigma|}{M}\Big)^3.
\]
For $\frac{|\tilde T-\Sigma|}{M|\nu_a|}\leq \Lambda^{-\frac{(1-\epsilon)}{2}}$ we therefore have $|e^{i\lambda \eta (G-G_0)(\Sigma,\tilde T,\tilde X)}-1|\lesssim \Lambda\Big( \frac{|\tilde T-\Sigma|}{M}\Big)^3\lesssim \Lambda^{-(1-3\epsilon)/2}$.
As 
\[
J|B(\lambda\eta E(\tilde E_c(\Sigma))^{3/2})-B(\lambda\eta)|\sim \frac{J}{\lambda}|\tilde E_c|\sim\frac{J}{\lambda}\frac{|\tilde T-\Sigma|}{M},
\]
we get $|e^{iJ(B(\lambda\eta E(\tilde E_c(\Sigma))^{3/2})-B(\lambda\eta))}-1|\lesssim \frac{J}{\lambda}\Lambda^{-(1-\epsilon)/2}$. Therefore, with $\tilde P:=\Lambda^{\frac{1-\epsilon}{2}}(\tilde T-\Sigma)/(M\nu_a)$, the integral in the second line of \eqref{equaint} is bounded with
\[
M\Lambda^{-\frac{1-\epsilon}{2}}\int_{|\tilde P|\leq 1}\Big(\Lambda (\tilde P \Lambda^{-\frac{1-\epsilon}2})^3+\frac{J}{\lambda}(\tilde P\Lambda^{-\frac {1-\epsilon}2)})+\frac{J}{\Lambda\lambda}\Big)d\tilde P\lesssim \frac{M}{\Lambda^{1-\epsilon}}(\Lambda^{\epsilon}+\frac{J}{\lambda})\sim \frac{M}{\Lambda^{1-2\epsilon}}
\,,
\]
as $J\leq M\ll \lambda$, $\Lambda=\lambda M$. We conclude as $ \frac{M}{\Lambda^{1-2\epsilon}}=\frac{M^{2\epsilon}}{\lambda^{1-2\epsilon}}$.
\end{proof}
%
%Recall the form \eqref{eq:73} of $G$ whose (degenerate) critical point satisfies
%\[
%\Sigma^2+\tilde X-2\gamma (\tilde T-\Sigma)\Big(1+O(\gamma(\tilde T-\Sigma))\Big)=0, \quad \Sigma+\gamma \Big(1+O(\gamma (\tilde T-\Sigma))\Big)=0,
%\]
%where there is no $\tilde X$ variable in the second equation (since the phase is linear in $\tilde X$). Since $M$ is large and $|\tilde T|\leq \varepsilon_0$, the last equation has an unique solution $\Sigma_c=\Sigma_c(\gamma, \tilde T)= -\gamma \Big(1+O(\gamma\tilde T)\Big)$. It follows that, for $\Sigma$ near $\Sigma_c$, the difference
%\[
%\lambda \eta(G(\Sigma, \tilde T,\tilde X)-G_0(\Sigma,\tilde T,\tilde X)) =\lambda \eta O(\gamma^2(\tilde T-\Sigma)^3)
%\]
%can be incorporated innto the symbol $\sigma_{C}$: $e^{i\lambda \eta(G(\Sigma, \tilde T,\tilde X)-G_0(\Sigma,\tilde T,\tilde X)) }$ does not oscillate. 

%Since it is more convenient to work with an explicit phase function we do take advantage of this remark and work with $G_0$ instead of $G$, since it became now clear that it represents the only contribution of the phase that matters.
%\end{rmq}

\subsection{Choice of $M$. Dispersive and Strichartz norms for $U$} 

Let $M\geq 4\lambda^{1/3}/c$ and $|\tilde T|\leq \sqrt{M/\lambda}$, as well as $|\tilde X|\leq \lambda^{-2/3}$, $|\tilde Y|:=|Y-4J/3|\lesssim \lambda^{-1}$; from \eqref{U} we get
\begin{equation}
  \label{eq:75}
 h |U((4J+2\tilde T)\sqrt{1+a},1+\tilde X,\frac 43J+ \tilde Y)|\gtrsim \lambda^{-1/3} \,.
\end{equation}
Recall that in the sum over $N$ defining $U$ there are at most $M_a$ terms : summing over $M_a$ intervals $I_{k}$ of size $\sqrt{M/\lambda}$ gives
\begin{equation}
  \label{eq:76}
  \|h U\|_{L^{q}(0,M_a; L^{\infty}_{X,Y})} \gtrsim (M_a\sqrt{M/\lambda})^{1/q}\lambda^{-1/3}\,.
\end{equation}
Asking moreover $M_a\sqrt{M/\lambda}\geq 1$ gives $M_a^2\geq \lambda/M$, $M\geq M_a$.
Recalling \eqref{eq:41bis} and \eqref{eq:56}, we get a condition on $\lambda$:
\begin{equation}
  \label{eq:77}
 (M_a\sqrt{M/\lambda})^{1/q}  \lambda^{-1/3}\lesssim \lambda^{1-1/q} M_{a}^{(1/2-2/q)}\lambda^{-5/4} M^{\frac 14}
\end{equation}
and it turns out that the best choice of parameters in order to maximize $q$ is $a\sim h^{1/3}$,  $M\sim M_a\sim \lambda^{1/3}$ which yields, for large $\lambda$, $q\geq 5$, e.g. $\alpha\leq 2/5$ and a loss $\beta\geq 1/10$ at the endpoint $(5,\infty)$.

\begin{figure}
        \begin{center}
                \includegraphics[width=15cm]{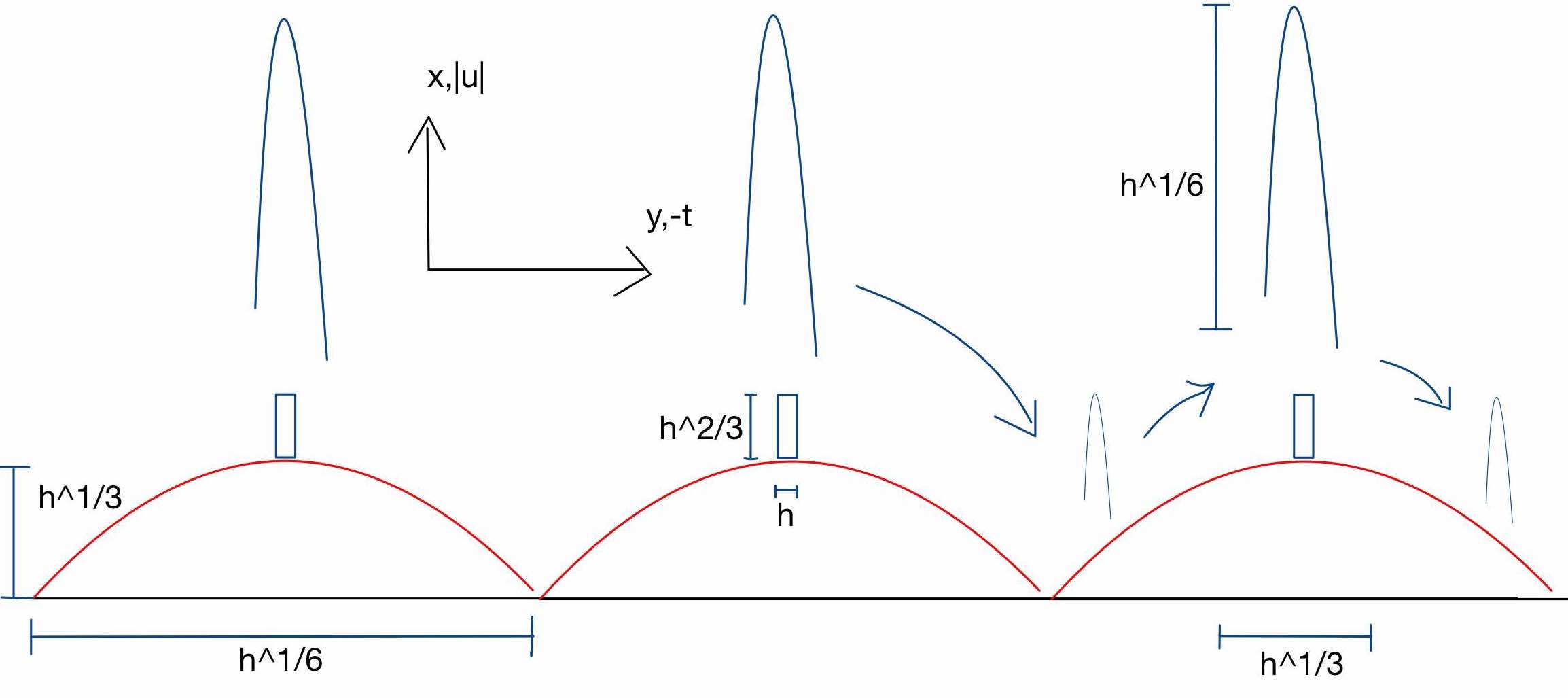}
        \end{center}
        \caption{Wave packet scales in space-time}
        \label{rays}
\end{figure}

We now compute the $L^{r}_{X,Y}$ norms for $r<+\infty$ while retaining the chosen values of $a$ and $M$: for $|\tilde T|\lesssim \lambda^{-1/3}$, $T=4J+\tilde T$, we get that
\begin{align*}
  \int_{X,Y} | h U(T,X,Y)|^{r}\, dXdY  & \geq   \int_{|\tilde X|\leq \lambda^{-2/3},|Y-\frac 43 J|\lesssim \lambda^{-1}} |h U(T,X,Y)|^{r}\, dXdY\\
  & \gtrsim  \int_{|\tilde X|\leq \lambda^{-2/3},|Y-\frac 43 J|\lesssim \lambda^{-1}} \lambda^{-r/3}\, dXdY\\
     & \gtrsim  \lambda^{-5/3-r/3}\,.
\end{align*}
Then, we have
\begin{align*}
  \int_{I_{J}} \left(\int_{X,Y} | h U(T,X,Y)|^{r}\, dXdY\right)^{\frac q r} dT  & \geq \int_{|\tilde T|\lesssim \lambda^{-\frac 13}} \left(\int_{X,Y} |h U(\cdots)|^{r}\, dXdY\right)^{\frac q r} dT\\
                                                                                &    \gtrsim \lambda^{-1/3} \lambda^{-5q/(3r)-q/3}\,,
\end{align*}
and
\begin{align*}
\sum_{J}\int_{I_{J}} \left(\int_{X,Y} |h U(T,X,Y)|^{r}\, dXdY\right)^{\frac q r} dT            & \gtrsim  M_{a}\lambda^{-1/3} \lambda^{-5q/(3r)-q/3}\,,\\
\int_0^{M_{a}} \left(\int_{X,Y} |h U(T,X,Y)|^{r}\, dXdY \right)^{\frac q r}dT            & \gtrsim \lambda^{-5q/(3r)-q/3}\,.
\end{align*}
Recalling \eqref{eq:41bis}, we get
$$
\lambda^{-\frac 5{3r}-\frac 1 3}\lesssim \lambda^{1-\frac 1 q-\frac 2 r} \lambda^{\frac 1 3 (\frac 1 2 -\frac 1 r -\frac 2 q)} \lambda^{-\frac 5 4} \lambda^{\frac 1 {12}}
$$
which translates into
$$
\frac 5 q +\frac 2 r-1\leq 0\,,
$$
which is our statement \eqref{adm-1/12}. This proves Theorem \ref{thm-1}.\qed

We now consider a different choice of parameters: assume $a\sim h^{1/2-\varepsilon}$ and retain $M\sim M_{a}=h^{-1/4-\varepsilon/2}$. Then $M\sim h^{2\varepsilon}\lambda$, we still have $|\tilde X|\lesssim \lambda^{-2/3}$ and $|\tilde Y|\lesssim \lambda^{-1}$, but now $|\tilde T|\lesssim h^{\varepsilon}$. Therefore we now get a condition on $\lambda$ that reads
$$
\lambda^{\frac 1 q}h^{\varepsilon/q} \times \lambda^{-\frac 5{3r}-\frac 1 3}\lesssim \lambda^{1-\frac 1 q-\frac 2 r} \lambda^{\frac 1 2 -\frac 1 r -\frac 2 q}  \lambda^{-\frac 5 4} \lambda^{\frac 1 {4}}h^{2\varepsilon(\frac 3 4 -\frac 1 r -\frac 3q)}\,;
$$
this turns out to match exactly the requirements from \cite{doi} (see Remark 1.7): for $(q,r)$ such that $r\geq 4$, we necessarily have, with a positive $C(\cdot)$ in the meaningful range,
$$
   \frac 3 q +\frac 1 r \leq \frac {15} {24}-\varepsilon C(\varepsilon,q,r)\,.
$$
 One may take $\varepsilon$ to zero and rewrite this condition on $(q,r)$ to highlight its distance to the free space requirement:
 \begin{equation}
   \label{eq:doi2d}
   \frac 1 q \leq  \left( \frac 1 2 -\frac {1-4/r}{12-24/r}\right) \left(\frac 1 2 -\frac 1 r\right)
 \end{equation}
 making clear the restriction $r\geq 4$ to be relevant as well as the loss ($1/12$ for the $(q,\infty)$ pair, e.g. $q\geq 24/5>4$.)
 \subsection{Higher dimensions}

In this section we prove Theorem \ref{thm-2} by taking advantage of the 2D example we just constructed: consider for simplicity, for $d\geq 3$, the isotropic model convex domain   $\Omega_d=\{(x,y)\in\mathbb{R}^{1+d}|, x>0,y\in\mathbb{R}^{d}\}$ and
$\Delta_F=\partial^{2}_{x}+(1+x)\Delta_{y}$ with Dirichlet boundary condition (one may without loss of generality replace $x\Delta_{y}$ by $xQ(y)$ where $Q$ is a constant coefficient second order elliptic operator). Denote by $u(t,x,y_{1})$ the solution to the 2D equation we previously constructed (in unscaled variables), and let $\phi$ be a smooth function from $\mathbb{R}^{d-2}$ to $\mathbb{R}$  such that $\hat \phi$ is positive, has compact support in a ball of size one and $\hat \phi=1$ near the origin. We may moreover select such a bump function so that, for $|y'|\leq 1$ $\phi(y')\geq 1/10$. Set $\phi_{h}(y')= h^{-(d-2)/4}\phi(y'/\sqrt h)$, which is $L^{2}-$normalized.

We seek a solution to the $d-$dimensional wave equation of the form
$$
v(t,x,y_{1},y')=u(t,x,y_{1}) \phi_{h}(y')+w(t,x,y)\,,
$$
with $w(0,x,y)=0$. Plugging our ansatz into the wave equation, we get
$$
(\partial_{t}^{2} -\Delta_{F}) w+\phi_{h}(y')(\partial_{t}^{2} -(\partial_{x}^{2}+(1+x)\partial^{2}_{y_{1}}) u -u(t,x,y_{1}) (1+x)\Delta_{y'} \phi_{h}=0\,.
$$
The middle term vanishes: $u$ is a solution to the 2D wave equation, and $\Delta_{y'} \phi_{h}(y')= \tilde \phi_{h}(y') /h$, where $\tilde \phi_{h}$ is again $L^{2}-$normalized. Therefore,
$$
(\partial_{t}^{2} -\Delta_{F}) w= \frac 1 h  u(t,x,y_{1}) (1+x)\tilde \phi_{h}(y')\,.
$$
If we denote by $F$ the source term, $w$ is the solution given by the Duhamel formula: if the wave equation satisfies an homogeneous Strichartz estimate with exponents $(q,r)$, then
\begin{equation}\label{stricrdinhomo} h^{\beta}\|\chi(hD_t)w\|_{L^q([0,T],
L^r)}\lesssim  h \int_{0}^{T}\|F\|_{L^{2}_{x,y}}\,,
\end{equation}
and therefore we have
\begin{align*}
h^{\beta}\|\chi(hD_t)w\|_{L^q([0,T],
  L^r)} & \lesssim  T \sup_{t}\| u(t,x,y_{1}) (1+x)\tilde \phi_{h}(y')\|_{L^{2}_{x,y}}\\
  & \lesssim  T \| u(0,x,y_{1}) \|_{L^{2}_{x,y_{1}}}\,.
\end{align*}
We are left with computing the $L^{r}_{y'}$ norm of $ \phi_{h}$: from its construction, we have $\|\phi\|_{L^{r}}\sim 1$ and by rescaling,
$$
\|\phi_h\|_{L^{r}}\sim h^{-\frac{d-2} 4} h^{\frac{d-2} {2r}}=h^{\frac{d-2} 2 (\frac{1} {r}-\frac 1 2)}\,.
$$
  From
  $$
  \|u\|_{L^{q}_{t}L^{r}_{x,y_{1}}} \|\phi_{h}\|_{L^{r}_{y'}} -\|w\|_{L^{q}_{t}L^{r}_{x,y}}\leq \|v\|_{L^{q}_{t}L^{r}_{x,y}}
  $$
  and our computation from the 2D case in the case where $a\sim h^{1/2-\varepsilon}$, together with $\lambda \sim h^{-1/4-3\varepsilon/2}$, we eventually get the limiting condition
  $$
2(d-2)(\frac 1 2-\frac 1 r)- \frac 4{q}-\frac{4}{3r}+\frac 5 6\leq 0
$$
in other words
$$
\frac 1 q \leq \left( \frac {d-1} 2 -\frac {1-4/r}{12-24/r}\right) \left(\frac 1 2 -\frac 1 r\right)\,,
$$
which is the desired condition.

\begin{rmq}
    The general philosophy is that of the usual Knapp counterexample: the main propagation is in the direction $y_{1}$, and our wave packet has no time to decorrelate in transverse directions. A similar argument was used in \cite{doi2} to extend the previous counterexamples from the 2D model to the general case of any strictly convex domain in higher dimensions.
  \end{rmq}
  \begin{rmq}
    If one plugs the other case, $a\sim h^{1/3}$ and $M\sim \lambda^{1/3}$ in the higher dimensional setting, it does not provide any interesting condition. The main difference appears to be that in that later case, the 2D counterexample is reaching its peak on very small subintervals (size $\lambda^{-1/3}$) whereas in the limiting sense, for $a\sim h^{1/2}$ and $M\sim \lambda$, the constructed example maintains its peak on the whole time interval, like the usual Knapp counterexample.
  \end{rmq}

  \appendix
  \section{Complements on Airy and related functions}
Well-known properties of Airy functions, including $A_{\pm}$ may be found on classical textbooks on special functions. The recent reference \cite{AFbook} provides an extensive review of such functions.
\subsection{Proof of Lemma \ref{lemL}}
From $A_+$ being analytic with values in $\mathbb{C}$ and never vanishing on the real line, there exist unique analytic functions $\rho(\omega)>0$ and $\theta(\omega)\in\mathbb{R}$ 
such that $A_+(\omega)=\rho(\omega)e^{i\theta(\omega)}$. Then one has $A_-(\omega)=\rho(\omega)e^{-i\theta(\omega)}$ and, from its definition, $L(\omega)=\pi+2\theta(\omega)$ is real on the real axis. 
Using \eqref{eq:Apm}, at $\omega=0$ we find $A_{\pm}(0)=e^{\mp i\pi/3}Ai(0)$ which yields $\frac{A_-(0)}{A_+(0)}=e^{-2i\theta(0)}=e^{2\pi i/3}$, hence $L(0)=\pi+2\theta(0)=\frac{\pi}{3}$. As $Ai(-\omega)=\sum_{\pm}A_{\pm}(\omega)=2\rho(\omega) \cos(\theta(\omega))$, \cite[(2.87), (2.104), (2.106)]{AFbook} yields asymptotic expansions for $\theta$ and $\rho$ as $\omega\rightarrow +\infty$:
\begin{gather}
\omega^{1/2}(2\rho(\omega))^2\sim_{\frac 1\omega} \frac{1}{\pi}\sum_{k=0}^{\infty}\alpha_k \omega^{-3k}, \quad \alpha_0=1,\\
\theta(\omega)+\frac{\pi}{4}\sim_{\frac 1\omega} \frac 23 \omega^{3/2}\sum_{k=0}^{\infty}\beta_k\omega^{-3k},\quad \beta_0=1,\quad \beta_1=-5/32.
\end{gather}
The last formula yields $L(\omega)=\pi+2\theta(\omega)=\frac 43 \omega^{3/2}+\frac{\pi}{2}-B(\omega^{3/2})$, where we set
\[
B(\omega^{3/2})\sim_{\frac 1\omega}-\frac 43\omega^{3/2}\sum_{k\geq 1}\beta_k \omega^{-3k}\,.
\]
Setting $b_{2k-1}:=-\frac 43\beta_k$ and $b_{2k}=0$ yields \eqref{eq:B} with $b_1=\frac{5}{24}>0$. From \cite[(2.95)]{AFbook} we have moreover $(2\rho(\omega))^2\theta'(\omega)=\frac{1}{\pi}$ : this yields $L'(\omega)=2\theta'(\omega)=\frac{1}{2\pi\rho^2(\omega)}>0$,
%As $A_{\pm}$ are linearly independant solutions to the same second order ODE, the Wronskian $A'_+(\omega)A_-(\omega)-A'_-(\omega)A_+(\omega)$ is constant and the following identity holds
%\begin{align}\label{LL}
%L'(\omega)&=i\frac{A_+}{A_-}(\omega)\Big(\frac{A'_-}{A_+}(\omega)-\frac{A_-A'_+}{A_+^2}(\omega)\Big)\\
%&=-\frac{i}{A_+(\omega)A_-(\omega)}(A'_+(\omega)A_-(\omega)-A'_-(\omega)A_+(\omega)) \nonumber \\
%&=\frac{c_0}{ \rho^2(\omega)}, \quad c_0=-\sqrt{3}Ai'(0)Ai(0)>0.\nonumber
%\end{align}
%Indeed, the Wronskian has value
%\begin{align*}
%W(\omega) & :=A'_+(\omega)A_-(\omega)-A'_-(\omega)A_+(\omega)\\
%& =e^{-i\pi/3}Ai'(e^{-i\pi/3}\omega)Ai(e^{+i\pi/3}\omega)-e^{i\pi/3}Ai'(e^{i\pi/3}\omega)Ai(e^{-i\pi/3}\omega)\\
% & =  W(0) =(e^{-i\pi/3}-e^{i\pi/3})Ai'(0)Ai(0)\\
% &  =-2i \sin(\pi/3)Ai'(0)Ai(0)=-i\sqrt{3}Ai'(0)Ai(0)\,.
%\end{align*}
hence $L$ is strictly increasing. Set $A(\omega)=Ai(-\omega)$, then $A(\omega)=2\rho(\omega)\cos(\theta(\omega))$. Therefore, $A(\omega)=0$ is equivalent to $\theta(\omega)=\pi/2+l\pi$, $l\in \mathbb{Z}$, which is equivalent to $L(\omega)=2\pi(1+l)$. From $L$ being a diffeomorphism from $\mathbb{R}$ onto $(0,\infty)$, one has that for all integer $k\geq 1$, $Ai(-\omega_k)=0$ if and only if $L(\omega_k)=2\pi k$.

Finally, using the Airy equation $A''(z)+z A(z)=0$ and integrating by parts, we get
\begin{align*}
\int_{0}^{\infty}Ai^2(x-\omega)dx=\int_{-\infty}^{\omega}A^2(z)dz & =\omega A^2(\omega)-\int_{-\infty}^{\omega}2zA(z)A'(z)dz\\
 & =\omega A^2(\omega)+\int_{-\infty}^{\omega}2A''(z)A'(z)dz\\
 & =\omega A^2(\omega)+A'^2(\omega).
\end{align*}
From $A'(\omega_k)=2\rho'(\omega_k)\cos(\theta(\omega_k))+2\rho(\omega_k)\theta'(\omega_k)\sin(\theta(\omega_k))$, as well as $A(\omega_k)=2\rho(\omega_k)\cos(\theta(\omega_k))=0$ (therefore $\sin(\theta(\omega_k))\in \{\pm1\}$), we get using $(2\rho(\omega))^2\theta'(\omega)=\frac{1}{\pi}$, $L'(\omega)=2\theta'(\omega)$
\[
\int_0^{\infty}Ai^2(x-\omega_k)dx=A'^2(\omega_k)=Ai'^2(-\omega_k)=4\rho^2(\omega_k)\theta'^2(\omega_k)=\frac{\theta'(\omega_k)}{\pi}=\frac{L'(\omega_k)}{2\pi},
\]
%From $2\pi Ai(0)=3^{-1/6}\Gamma(1/3)$, $2\pi Ai'(0)=-3^{1/6}\Gamma(2/3)$ and the Euler formula for the $\Gamma$ function, $\Gamma(x)\Gamma(1-x)=\frac{\pi}{\sin(\pi x)}$, we get $2\pi c_0=1$, 
thus the last assertion in \eqref{eq:propL2} holds true. The proof of Lemma \ref{lemL} is complete.\qed
\subsection{Proof of Lemma \ref{lemorthog}}
Using the Airy equation we juste recalled,  one easily checks that $(e_k)_{k}$ are the eigenfunctions of $-\partial^2_x+(1+x)\theta^2$ with Dirichlet boundary condition at $x=0$, associated with eigenvalues $\lambda_k(\theta)$. It will be enough to prove that they form an orthogonal family on $L^2(\mathbb{R}_+)$. In order to do so, we use well-known formulas for Airy functions : in particular, it follows from \cite[(3.53)]{AFbook}) that
\begin{multline*}
    \frac{d}{dx} \Big[\frac{Ai'(\alpha(x+\beta_1))Ai(\alpha(x+\beta_2))-Ai(\alpha(x+\beta_1))Ai'(\alpha(x+\beta_2)))} {\alpha^2(\beta_1-\beta_2)}\Big]\\ =  Ai(\alpha(x+\beta_1))Ai(\alpha(x+\beta_2))\,.
\end{multline*}
%\begin{multline}
%\int Ai(\alpha(x+\beta_1))Ai(\alpha(x+\beta_2))dx=\\
%\frac{1}{\alpha^2(\beta_1-\beta_2)}\Big[Ai'(\alpha(x+\beta_1))Ai(\alpha(x+\beta_2))-Ai(\alpha(x+\beta_1))Ai'(\alpha(x+\beta_2)))\Big].
%\end{multline}
Taking $\alpha=1$, $\beta_1=-\omega_k$, $\beta_2=-\omega_j$ with $k\neq j$, we can therefore compute explicitely $\int_{0}^{+\infty} Ai(x-\omega_{k})Ai(x-\omega_{j})\,dx$  to get vanishing traces at $x=0$ and $x=+\infty$, hence
\[
\langle e_k,e_j \rangle_{L^2(\mathbb{R}_+)}=0\,,
\]
and this holds for all $k(\neq j)$.\qed

\subsection{Proof of Lemma \ref{AiryPoisson}}
Consider $\phi(\omega)$, a smooth, compactly supported, function defined for $\omega \in \R$. The function $L(\omega)$ defines a one to one map from $\R$ to $\R_{+}$. Now define $\varphi(x)$ for $x\in \R_{+}$ with
$$
\varphi(L(\omega))=\frac 1 {L'(\omega)} \phi (\omega)\,.
$$
We may extend $\varphi$ to be zero for $x\in \R_{-}$ and still retain a smooth, compactly supported function:  there exists $\omega^{\flat}$ such that $\phi(\omega)=0$ if $\omega<\omega^{\flat}$, and we have $\varphi(x)=0$ for $x<L(\omega^{\flat})$, e.g. $\varphi$ is always supported on $\R_{+}^{*}$. We apply the usual Poisson sumation formula to $\varphi$, which reads
$$
\sum_{k\in \Z} \varphi(2\pi k)=\frac 1 {2\pi} \sum_{N\in \Z} \int_{\R} e^{-i N x}\varphi(x)\,dx\,.
$$
Since $\varphi$ vanishes on $\R_{-}$, this becomes
$$
\sum_{k\in \N^{*}} \varphi(2\pi k)=\frac 1 {2\pi} \sum_{N\in \Z} \int_{\R_{+}} e^{-i N x}\varphi(x)\,dx\,,
$$
and we can now change variables with $x=L(\omega)$:
$$
\sum_{k\in \N^{*}} \varphi(2\pi k)=\frac 1 {2\pi} \sum_{N\in \Z} \int_{\R} e^{-i N L(\omega)}\varphi(L(\omega))L'(\omega)\,d\omega\,,
$$
and recalling that $L(\omega_{k})=2\pi k$ this reads
$$
\sum_{k\in \N^{*}} \varphi(L(\omega_{k}))=\frac 1 {2\pi} \sum_{N\in \Z} \int_{\R} e^{-i N L(\omega)}\varphi(L(\omega))L'(\omega)\,d\omega\,.
$$
Finally, with $\varphi\circ L=\phi/L'$, 
$$
\sum_{k\in \N^{*}} \frac {\phi(\omega_{k})}{L'(\omega_{k})}=\frac 1 {2\pi} \sum_{N\in \Z} \int_{\R} e^{-i N L(\omega)}\phi(\omega)\,d\omega\,,
$$
which is the desired formula. \qed

\def\cprime{$'$} \def\cprime{$'$}

\end{document}